\documentclass[a4paper,11pt,twoside,reqno]{amsart} 
\DeclareSymbolFontAlphabet{\mathcal}{symbols}
\usepackage[colorlinks=red,linkcolor=blue]{hyperref}
\usepackage[cmtip,all]{xy}
\usepackage{hyperref}
\usepackage{amsmath}
\usepackage[psamsfonts]{amssymb}
\usepackage[psamsfonts]{eucal}
\usepackage{amsthm}
\usepackage[psamsfonts]{amssymb}
\usepackage[scaled]{helvet}
\usepackage{mathrsfs}
\usepackage{bm}
\usepackage[a4paper]{geometry}
\usepackage{multicol}
\usepackage{enumerate}
\usepackage{graphicx}
\usepackage{tikz}
\usepackage{color}
\definecolor{trama}{gray}{.875}
\usepackage{lmodern}
\usepackage{textcomp}% symboles supplémentaires
\usepackage{enumerate}
\usepackage{wrapfig}
\title[Poincar\'e duality and intersection homology]{Poincar\'e duality with cap products in intersection homology }
\date{\today}

\author{David Chataur}
\address{Lafma\\
Universit\'e de Picardie Jules Verne\\
33, rue Saint-Leu\\
80039 Amiens Cedex~1\\
         France}
\email{David.Chataur@u-picardie.fr}

\author{Martintxo Saralegi-Aranguren}
\address{Laboratoire de Math{\'e}matiques de Lens\\  
      EA 2462 \\
      Universit\'e d'Artois\\
         SP18, rue Jean Souvraz\\
          62307 Lens Cedex\\
         France}
\email{martin.saraleguiaranguren@univ-artois.fr}

\author{Daniel Tanr\'e}
\address{D\'epartement de Math{\'e}matiques\\
         UMR 8524 \\
         Universit\'e de Lille~1\\
         59655 Villeneuve d'Ascq Cedex\\
         France}
\email{Daniel.Tanre@univ-lille1.fr}

\thanks{
The third author was  supported by the MINECO and FEDER research project MTM2016-78647-P. 
and the ANR-11-LABX-0007-01  ``CEMPI''}

\subjclass[2010]{55N33, 57P10, 57N80, 55U30}

\keywords{Intersection homology; cap product; cup product; Poincar\'e duality}

\makeatletter
\renewcommand\l@subsection{\@tocline{2}{0pt}{2pc}{5pc}{}}
\renewcommand\l@subsubsection{\@tocline{3}{0pt}{4pc}{10pc}{}}
\makeatother
\theoremstyle{plain}
\newtheorem{theorem}{Theorem}

\newtheorem{theoremv}{Main Theorem}

\newtheorem{proposition}{Proposition}[section]
\newtheorem{theoremb}[proposition]{Theorem}
\newtheorem{lemma}[proposition]{Lemma}
\newtheorem{corollary}[proposition]{Corollary}

\theoremstyle{definition}
\newtheorem{definition}[proposition]{Definition}
\newtheorem{example}[proposition]{Example}

\theoremstyle{remark}
\newtheorem{remark}[proposition]{Remark}

\newcommand{\secref}[1]{Section~\ref{#1}}

\newcommand{\thmref}[1]{Theorem~\ref{#1}}
\newcommand{\propref}[1]{Proposition~\ref{#1}}
\newcommand{\lemref}[1]{Lemma~\ref{#1}}
\newcommand{\corref}[1]{Corollary~\ref{#1}}

\newcommand{\remref}[1]{Remark~\ref{#1}}

\newcommand{\defref}[1]{Definition~\ref{#1}}

\def\R{{\mathbb R}}

\def\ov{\overline}

\def\cA{{\mathcal A}}
\def\cAb{{\mathcal Ab}}
\def\cB{{\mathcal B}}
\def\cC{{\mathcal C}}
\def\cD{{\mathcal D}}

\def\cF{{\mathcal F}}
\def\cG{{\mathcal G}}
\def\cH{{\mathcal H}}

\def\cK{{\mathcal K}}
\def\cL{{\mathcal L}}

\def\cR{{\mathcal R}}
\def\cS{{\mathcal S}}
\def\cU{{\mathcal U}}
\def\cV{{\mathcal V}}
\def\cW{{\mathcal W}}
\def\cX{{\mathcal X}}

\def\crH{{\mathscr H}}

\def\1{{\mathbf 1}}
\def\gd{{\mathfrak{d}}}

\def\gC{{\mathfrak{C}}}
\def\gH{{\mathfrak{H}}}

\def\gT{{\mathfrak{T}}}

\def\tc{{\mathtt c}}

\def\tv{{\mathtt v}}
\def\tw{{\mathtt w}}

\def\N{\mathbb{N}}
\def\Q{\mathbb{Q}}
\def\R{\mathbb{R}}
\def\Z{\mathbb{Z}}

\def\Ker{{\rm Ker\,}}

\def\im{{\rm Im\,}}

\def\Hom{{\rm Hom}}
\def\Supp{{\rm Supp\,}}

\def\id{{\rm id}}

\def\codim{{\rm codim\,}}

\def\sing{{\rm sing}}

\def\pr{{\rm pr}}
\def\tN{{\widetilde{N}}}
\def\tT{{\widetilde{\mathrm{T}}}}

\def\reg{{\rm reg}}

\def\ffss{{filtered face sets}}

\def\rc{{\mathring{\tc}}}

\def\tDelta{{\widetilde{\Delta}}}

\def\tg{{\tilde{g}}}
\def\si{{\mathtt{Simp}}}
\def\tto{{\mathtt{o}}}
\def\tors{{\mathrm{Tors}\,}}

\begin{document}

\begin{abstract}
For having a Poincar\'e duality via a cap product
between the intersection homology of a paracompact oriented pseudomanifold and the cohomology given
by the dual complex, 
G. Friedman and J. E.  McClure need a  coefficient field or an additional hypothesis on the torsion. 
In this work, by using
the classical geometric process of blowing-up, adapted to a simplicial setting, we build a cochain complex which gives
a Poincaré duality via a cap product with intersection homology,  
for any commutative ring of coefficients.
We prove also the topological invariance of the
blown-up intersection cohomology with compact supports
in the case of a paracompact pseudomanifold with no codimension one strata.

This work is written with general perversities, defined on each stratum and not only in function of the codimension of strata.
It contains also  a tame intersection homology, suitable for large perversities.
\end{abstract}

\maketitle

\tableofcontents

%%%%%%%%%%%%%%%%%%%%%%%%%
\section*{Introduction}

Intersection homology was defined by
M.~Goresky and R.~MacPherson
in \cite{MR572580},
 \cite{MR696691}, with the existence of a Poincar\'e duality in the case of rational coefficients.
If $X$ is a compact, oriented, $n$-dimensional PL-pseudomanifold,
Goresky and MacPherson establish in their first paper on intersection homology
(\cite[Theorem 1]{MR572580}, see also \cite{MR2529162}, \cite{2016arXiv160905975F}) the existence of an intersection product,
$\pitchfork\colon H_{i}^{\ov{p}}(X;\Z)
\times  H_{j}^{\ov{q}}(X;\Z)\to H_{i+j-n}^{\ov{r}}(X;\Z)$, for perversities such that
$\ov{p}+\ov{q}\leq \ov{r}$.
Let $\ov{t}$ be the top perversity defined by $\ov{t}(i)=i-2$.
By composing with an augmentation, $\varepsilon\colon H_{0}^{\ov{t}}(X;\Z)\to\Z$, 
the authors show 
(\cite[Theorem 2]{MR572580}) 
that this correspondence gives a  bilinear form,
$$H_{i}^{\ov{p}}(X;\Z)
\times  H_{n-i}^{\ov{t}-\ov{p}}(X;\Z)
\to H_{0}^{\ov{t}}(X;\Z)\xrightarrow{\varepsilon} \Z,$$
which is non degenerate after tensorisation with $\Q$. 
As shown by Goresky and Siegel in \cite{MR699009}, an extension of this result to $\Z$
cannot remain without an hypothesis on the torsion of the intersection homology 
of the links of the pseudomanifold
(see \cite{MR2507117} for an extension to homotopically stratified spaces).

We also mention the different approach of M. Banagl (\cite{MR2662593}) 
who
associates a CW-complex $I^{\ov{p}}X$ to certain stratified spaces. The
rational homology of these spaces
satisfies a generalized form of Poincar\'e duality and presents some concrete advantages.
Their homology being different from  intersection homology, their study needs an ad'hoc approach
and they are not considered in this work.

There exists also an approach to  Poincar\'e duality 
of a manifold by mixing homology and cohomology with a cap product.
This method was achieved with success in intersection homology and cohomology by
G.~Friedman and  J.E.~McClure (\cite{zbMATH06243610})
in the case of field coefficients, or with an hypothesis on the torsion of the intersection
homology of the links (\cite[Chapter 8]{FriedmanBook}). 
Their intersection cohomology is defined as the homology of the linear dual of the intersection chain complex;
we denote it by $H^*_{\ov{p}}(X;R)$ with $R$ a commutative ring.
In this context, the extension of such result to any commutative ring is not possible.

In this work, we continue with the paradigm of chain and cochain complexes.
But, as taking the linear dual is inappropriate for having a 
Poincar\'e duality cohomology-homology, we consider 
a cochain complex which takes more in account the singularities and overcomes
the restriction to coefficients in a field.
We obtain it as  a simplicial adaptation of the geometric  blow-up
which was already present in \cite{MR1171153}, \cite{MR1245833}.
More precisely, for any commutative ring $R$, we define a cochain complex endowed with a cup product,
$\tN^*_{\ov{\bullet}}(X;R)$,
whose homology in perversity $\ov{p}$ is denoted
$\crH^*_{\ov{p}}(X;R)$
and called {\it blown-up intersection cohomology} (or Thom-Whitney cohomology in some previous works,
\cite{CST1}, \cite{CST2}, \cite{CST3}, \cite{CST7}).
A version with compact supports is  introduced in \defref{def:support} 
and denoted $\crH^{*}_{\ov{p},c}(X;R)$.
In the case of  Goresky and MacPherson perversities (\cite{MR572580}), our main result can be stated as follows.

\begin{theoremv}\label{thm:dualPoincare}
Let $R$ be a commutative ring and $X$  an  oriented, paracompact,
 $n$-dimensional pseudomanifold. 
Then, for any Goresky and MacPherson perversity, the cap product with the orientation class
of $X$ 
defines an isomorphism
$$\cD\colon \crH^{i}_{\ov{p},c}(X;R)\xrightarrow[]{\cong}  H_{n-i}^{\ov{p}}(X;R).$$
\end{theoremv}

 The complex $\tN^*_{\ov{\bullet}}(X;R)$ 
has several properties which facilitate its use.
For instance, the complex $\tN^*_{\ov{\bullet}}(X;R)$ is local in essence and it allows the determination of
the admissibility of a cochain by considering individually each simplex of its support. 
We quote also  that the operations {\it cup} and {\it cap}
are defined from cochain complexes and not only in the derived category. 
The existence of $\text{cup}_i$ products at the cochain level
allowed in \cite{CST2} an explicit determination of the rank of perversities
in the definition of Steenrod intersection squares. As a consequence, we were able to
give a positive answer (\cite[Theorem B]{CST2}) to a conjecture of
M. Goresky and W. Pardon (\cite{MR1014465}).

Actually, we prove the Main Theorem in the  setting of general perversities introduced by
MacPherson in \cite{MacPherson90}, cf. \thmref{thm:dual}.
These perversities are defined individually on each stratum and not only as a function of  their codimension
(cf. \defref{def:perversite}).
This allows a larger spectrum of the values taken by the perversities.

Without going too much into details at the level of this introduction, we may observe that,
in the case of a  perversity $\ov{p}$ such that $\ov{p}\leq \ov{t}$,
each $\ov{p}$-allowable simplex as well as its boundary have a support
which is not included in the singular part.
As this property disappears if $\ov{p}\not\leq\ov{t}$, we introduce what we call
 {\it tame intersection homology} and denote $\gH_{*}^{\ov{p}}(X;R)$. %
 The tame intersection homology keeps the behavior of intersection homology 
 (see \cite{CST3}) and is isomorphic to it
 when $\ov{p}\leq\ov{t}$. We denote $\gH^*_{\ov{p}}(X;R)$ the associated cohomology
 and $\gH^*_{\ov{p},c}(X;R)$ the variant with compact supports.
 In the case of a paracompact oriented pseudomanifold, 
 \thmref{thm:dual} gives an isomorphism between the blown-up intersection cohomology
 $\crH^*_{\ov{p},c}(X;R)$
 and 
 $\gH_{n-*}^{\ov{p}}(X;R)$
 for any commutative ring $R$ and any perversity $\ov{p}$.
We complete this work with a proof of  the  topological invariance of
the blown-up cohomology with compact supports in \thmref{thm:invariance}.

\medskip
\secref{section:back} is a  review on intersection homology.
To achieve the program above, we define and establish the main properties of  the blown-up cohomology with compact supports, $\crH^*_{\ov{p},c}(-)$, in \secref{sec:suppcompact}:
existence of a Mayer-Vietoris sequence (\propref{prop:MVcompact}),
cohomology of a cone (\propref{prop:conecompact}),
cohomology of the product $X\times \R$ (\propref{prop:xproduitRcompact})
and comparison of $\crH^*_{\ov{p},c}(-)$ and $\gH^*_{\ov{p},c}(-)$ (\propref{prop:TWadroite}).
In particular, we prove $\crH^*_{\ov{p},c}(X;R)\cong \gH^*_{\ov{t}-\ov{p},c}(X;R)$
if $R$ is a field and $X$ a paracompact pseudomanifold.
The topological invariance for a paracompact CS set (see \defref{def:CSset})
with no codimension one strata and a 
Goresky and MacPherson perversity is 
established in \secref{sec:invariance} as \thmref{thm:invariance}.
 \secref{sec:dual} is concerned with the proof of Poincar\'e duality (\thmref{thm:dual}).
 In \thmref{thm:dualcup}, we specify the pairing obtained from the cup product
 of the blown-up cohomology
 and apply it to Witt spaces for the middle perversity.

\medskip
In all the text, $R$ is a commutative ring (always supposed with unit) and
we do not mention it explicitly in the proofs. 
The degree of an element $a$ of a graded module is represented by $|a|$.
For any topological space $X$, we denote by $\tc X=X\times [0,1]/X\times\{0\}$
the cone on $X$
and set $\rc_{t} X=X\times [0,t[/X\times\{0\}$
for any $t\in\R$ or $t=\infty$. \emph{The  open cone on $X$} corresponds to $t=1$ and is also denoted
$\rc X$ if there is no ambiguity.

%
%%%%%%%%%%%%%%%%%%%%%%%%%%%%%%%%%%%%%%%%
\section{Background on intersection homology}\label{section:back}

We recall the basics we need, sending the reader to \cite{MR572580},  \cite{FriedmanBook} or \cite{CST1} for more details.

\subsection{Pseudomanifolds}\label{subsec:filtres}

\begin{definition}\label{def:espacefiltré} 
A \emph{filtered space of dimension $n$,} $(X, (X_i)_{0\leq i\leq n})$, is a Hausdorff space
together with a filtration by closed subsets,
$$\emptyset=X_{-1}\subseteq X_0\subseteq X_1\subseteq\dots\subseteq X_n=X,$$
such that $X_n\backslash X_{n-1}\neq \emptyset$.
The connected components $S$ of $X_{i}\backslash X_{i-1}$ are the  \emph{strata} of $X$ and
we set $\dim S=i$, $\codim S=\dim X-\dim S$. 
The strata of $X_n\backslash X_{n-1}$  are called  \emph{regular.}    
The set of non-empty strata of $X$ is denoted  $\cS_X$. 
The subspace $X_{n-1}$ is called \emph{the singular set.} 
\end{definition} 

An open subset $U$ of $X$ is endowed with the \emph{induced filtration,} defined by 
$U_i = U \cap X_{i}$.
If $M$ is a manifold, the \emph{product filtration}
is defined by
$\left(M \times X\right) _i = M \times X_{i}$.  

\medskip
The CS sets introduced in \cite{MR0319207} are a weaker version of pseudomanifolds that is sufficient for the topological invariance property.

\begin{definition}\label{def:CSset}
A \emph{CS set} of dimension $n$ is a filtered space,
$$
X_{-1}=\emptyset\subseteq X_0 \subseteq X_1 \subseteq \dots \subseteq X_{n-2} \subseteq X_{n-1} \subsetneqq X_n =X,
$$
such that, for each $i\in\{0,\dots,n\}$, 
$X_i\backslash X_{i-1}$ is a topological manifold of dimension $i$ or the empty set. Moreover each $x \in X_i \backslash X_{i-1}$ with $i\neq n$ admits
\begin{enumerate}[(i)]
\item an open neighborhood $V$ of $x$ in $X$, endowed with the induced filtration,
\item an open neighborhood $U$ of $x$ in  $X_i\backslash X_{i-1}$, 
\item a filtered compact space $L$  of dimension  $n-i-1$, whose cone $\rc L$ is endowed with the conic filtration, 
$(\rc L)_{i}=\rc L_{i-1}$,
\item a homeomorphism, $\varphi \colon U \times \rc L\to V$, 
such that
\begin{enumerate}[(a)]
\item $\varphi(u,\tv)=u$, for any $u\in U$, where $\tv$  is the apex of $\rc L$,
\item $\varphi(U\times \rc L_{j})=V\cap X_{i+j+1}$, for any $j\in \{0,\dots,n-i-1\}$.
\end{enumerate}
\end{enumerate}
The filtered space $L$ is called the \emph{link} of $x$. 
The CS set is called \emph{normal} if its links are connected.
\end{definition}

We take over the original definition of pseudomanifold given by Goresky and MacPherson (\cite{MR572580}) 
but without the restriction on the existence of strata of codimension~1.

\begin{definition}\label{def:pseudomanifold}
A \emph{topological pseudomanifold of dimension $n$} (or a pseudomanifold) is a CS set of dimension $n$
whose links of points
$x\in X_{i}\backslash X_{i-1}$ are topological pseudomanifolds of dimension $(n-i-1)$
for all $i\in\{0,\dots,n-1\}$.
Any open subset of a pseudomanifold is a pseudomanifold for the induced structure.
\end{definition}

%%%%%%%%%%%%%%
\subsection{Perversities}\label{subsec:perversite}

We begin  with the perversities of \cite{MR572580} and continue with 
a more general notion of perversity, introduced in \cite{MacPherson90}
and already present in \cite{MR1245833}, \cite{MR2210257}, \cite{MR2721621}, 
  \cite{MR2796412},
\cite{zbMATH06243610}.

\begin{definition}\label{def:perversite} 
A \emph{{\rm GM}-perversity} is a map $\ov{p}\colon \N\to\Z$ such that
$\ov{p}(0)=\ov{p}(1)=\ov{p}(2)=0$ and
$\ov{p}(i)\leq\ov{p}(i+1)\leq\ov{p}(i)+1$, for all $i\geq 2$.
Among them, we mention the null perversity $\ov{0}$ constant with value~0 and the \emph{top perversity} defined by
$\ov{t}(i)=i-2$ if $i\geq 2$. 

A  \emph{perversity on a filtered space,} $(X,(X_{i})_{0\leq i\leq n})$, is an  application,
$\ov{p}_{X}\colon \cS_X\to \Z$, defined on the set of strata of $X$ and taking the value~0 on the regular strata.
The pair $(X,\ov{p}_{X})$ is called a \emph{perverse space} and denoted $(X,\ov{p})$ if there is no ambiguity.
(In the case of a CS set or a pseudomanilfold we use the expressions 
\emph{perverse CS set} and \emph{perverse pseudomanifold.})

If $\ov{p}$ and $\ov{q}$ are two perversities on $X$, we set  $\ov{p}\leq \ov{q}$ if we have
$\ov{p}(S)\leq\ov{q}(S)$, for all $S\in \cS_{X}$.
A GM-perversity induces a perversity on $X$ by $\ov{p}(S)=\ov{p}(\codim S)$.
For any perversity, $\ov{p}$, the perversity $D\ov{p}:=\ov{t}-\ov{p}$ is called the
\emph{complementary perversity} of $\ov{p}$.
\end{definition}

\subsection{Intersection Homology}\label{subsec:homologieintersection}

We specify the chain complex used for the determination of intersection homology,
cf. \cite{CST3}.

\begin{definition}\label{def:filteredsimplex}
Let  $X$ be a filtered space.
A  \emph{filtered simplex} is a continuous map $\sigma\colon\Delta\to X$, 
from a Euclidean simplex endowed with a decomposition
$\Delta=\Delta_{0}\ast\Delta_{1}\ast\dots\ast\Delta_{n}$,
 called \emph{$\sigma$-decomposition of $\Delta$},
such that
$$
\sigma^{-1}X_{i} =\Delta_{0}\ast\Delta_{1}\ast\dots\ast\Delta_{i},
$$
for all~$i \in \{0, \dots, n\}$. 
The sets  $\Delta_{i}$  may be empty, with the convention $\emptyset * Y=Y$, for any space $Y$. 
The simplex $\sigma$ is  \emph{regular} if $\Delta_{n}\neq\emptyset$. A chain is  \emph{regular} if it is a linear combination of regular simplices.
To connote that the filtration on $\Delta$ is induced from the filtration of $X$ by
$\sigma$, we  sometimes denote $\Delta=\Delta_{\sigma}$. 
\end{definition}

 \begin{definition}\label{def:lessimplexes}
 Let $(X,\ov{p})$ be a perverse space.
 The \emph{perverse degree of} a filtered simplex $\sigma\colon\Delta=\Delta_{0}\ast \dots\ast\Delta_{n} \to X$
  is the $(n+1)$-uple,
$\|\sigma\|=(\|\sigma\|_0,\dots,\|\sigma\|_n)$,  
where 
 $\|\sigma\|_{i}=
 \dim \sigma^{-1}X_{n-i}=\dim (\Delta_{0}\ast\dots\ast\Delta_{n-i})$, 
 with the  convention $\dim \emptyset=-\infty$.
For each stratum $S$  of $X$, the \emph{perverse degree of $\sigma$ along $S$} is defined by
  $$\|\sigma\|_{S}=\left\{
 \begin{array}{cl}
 -\infty,&\text{if } S\cap \sigma(\Delta)=\emptyset,\\
 \|\sigma\|_{\codim S},&\text{otherwise.}
  \end{array}\right.$$
 A \emph{filtered simplex is $\ov{p}$-allowable} if
  \begin{equation}\label{equa:admissible}
  \|\sigma\|_{S}\leq \dim \Delta-\codim S+\ov{p}(S),
  \end{equation}
   for each stratum $S$ of $X$. A chain $\xi$ is said
   \emph{$\ov{p}$-allowable} if it is a linear combination of $\ov{p}$-allowable simplices,
   and of  \emph{$\ov{p}$-intersection} if  $\xi$ together with its boundary are $\ov{p}$-allowable.
   We denote by
$C_{*}^{\ov{p}}(X;R)$ 
the complex  of  $\ov{p}$-intersection chains and by 
 $H^{\ov{p}}_*(X;R)$ its homology, called \emph{$\ov{p}$-intersection homology}.
 \end{definition}

In \cite[Théorème B]{CST3}, we prove that $H^{\ov{p}}_*(X;R)$ is naturally isomorphic to the intersection
homology of Goresky and MacPherson.

\begin{lemma}{\cite[Lemme 7.5]{CST3}}\label{lem:perversitepetite}
If the perversity $\ov{p}$ satisfies $\ov{p}\leq\ov{t}$, then any $\ov{p}$-allowable filtered simplex and its boundary are regular. % 
\end{lemma}

Notice that the hypothesis of \lemref{lem:perversitepetite} is satisfied for any GM-perversity.
On the contrary, if  $\ov{p}\not\leq\ov{t}$, some 
$\ov{p}$-allowable filtered simplices can be included in the singular part. 
As such simplex cannot be considered in the definition of the blown-up intersection cohomology 
(see the definition of the cap product in \secref{sec:suppcompact}), 
we adapt the definition of intersection homology to this situation
as follows.
First, we decompose the boundary of a filtered simplex $\Delta=\Delta_{0}\ast\dots\ast\Delta_{n}$  as
$$\partial\Delta=\partial_{\reg}\Delta+\partial_{\sing}\Delta,$$
where $\partial_{\reg}\Delta$ contains all the regular simplices.
In particular, we have
$$\partial_{\sing}\Delta=\left\{
\begin{array}{cl}
\partial\Delta&
\text{if } \Delta_{n}=\emptyset,\\
(-1)^{|\Delta|+1}\Delta_{0}\ast\dots\ast\Delta_{n-1}&
\text{if } \dim \Delta_{n}=0,\\
0&
\text{if } \dim \Delta_{n}>0.
\end{array}\right.$$
If $\sigma\colon \Delta\to X$ is a regular simplex, its boundary is decomposed  in
$\partial \sigma=\partial_{\reg}\sigma +\partial_{\sing}\sigma$.

\begin{definition}\label{def:adroite}
Let $(X,\ov{p})$ be a perverse space.
The chain complex $\gC_{*}^{\ov{p}}(X;R)$ is the $R$-module formed of the
regular $\ov{p}$-allowable chains whose boundary by $\partial_{\reg}$ is $\ov{p}$-allowable.
We call $(\gC^{\ov{p}}_{*}(X;R),\gd={\partial_{\reg}})$ the tame $\ov{p}$-intersection complex and its homology,
$\gH^{\ov{p}}_{*}(X;R)$, the
\emph{tame $\ov{p}$-intersection homology.} 
\end{definition}

Similar complexes have been already introduced by the second author in
\cite{MR2210257} and by G.~Friedman in \cite{MR2276609} and \cite[Chapter 6]{FriedmanBook}. 
In \cite{CST3}, we show that
$\gH^{\ov{p}}_{*}(X;R)$
is isomorphic to them. 
We recall now the  main properties of $\gH^{\ov{p}}_{*}(X;R)$ established in \cite{CST3}, see
also \cite[Chapter 6]{FriedmanBook}.

\begin{theoremb}{\cite[Propositions 7.10 and 7.15]{CST3}}\label{thm:adroite}
Let $(X,\ov{p})$ be a perverse space. The following properties are satisfied.
\begin{enumerate}[(1)]
\item If $\ov{p}\leq\ov{t}$, the intersection homology coincides with the tame intersection homology,
$$H^{\ov{p}}_{*}(X;R)= \gH^{\ov{p}}_{*}(X;R).$$
\item For any open cover $\cU=\{U,V\}$ of $X$,
there exists a  Mayer-Vietoris exact sequence,
$$\ldots\to \gH_{i}^{\ov{p}}(U\cap V;R)\to
\gH_{i} ^{\ov{p}}(U;R)\oplus \gH_{i} ^{\ov{p}}(V;R)\to
 \gH ^{\ov{p}}_i(X;R)\to \gH ^{\ov{p}}_{i-1}(U\cap V;R)\to\ldots
$$
\end{enumerate}
\end{theoremb}

\begin{proposition}{\cite[Corollaire 7.8]{CST3}}\label{prop:RfoisXAdr}
Let $(X,\ov{p})$ be a perverse CS set.  
Then the inclusions $\iota_{z}\colon X\hookrightarrow \R\times X$, $x\mapsto (z,x)$ with $z\in\R$ fixed, 
and the projection $p_{X}\colon \R\times X\to X$, $(t,x)\mapsto x$, induce isomorphisms, 
$\mathfrak H_{k}^{\ov{p}}(\R\times X;R)\cong  \mathfrak H_{k}^{\ov{p}}(X;R)$.
\end{proposition}

 \begin{proposition}{\cite[Proposition 7.9]{CST3}}\label{prop:homologieconeAdr}
Let $X$  be a compact filtered space of dimension~$n$. %  
We endow the cone $\rc X$ with a perversity $\ov{p}$ and with the the conic filtration, 
$(\rc X)_{i}=\rc X_{i-1}$. 
We denote also $\ov{p}$ the induced perversity on $X$.
Then, the tame $\ov{p}$-intersection homology of the cone is determined by,
$$
\gH_{k}^{\ov{p}}(\rc  X;R)\cong
\left\{
\begin{array}{ll}
\gH_{k}^{\ov{p}}(X;R) & \hbox{if }k< n-\ov{p}(  \tw ),\\
0& \hbox{if } k\geq n-\ov{p}( \tw ),
\end{array}
\right.
$$
where the isomorphism is induced by 
$\iota_{\rc X}\colon X\to \rc X$, $x\mapsto [x,t]$ with $t\in ]0,\infty[$, and $\tw$ is the apex of the cone.
\end{proposition}
%%%%%%%%%%%%%%%%%%%%%%%%%%%%%%%%%%%

%%%%%%%%%%%%%%%%%%%%%%%%%%%%%%%%
%%%%%%%%%%%%%%%%%%%%%%
\section{Blown-up intersection cohomology with compact supports}\label{sec:suppcompact}

In this section we recall the blown-up intersection cohomology of a perverse space and introduce its version with compact supports. We consider a filtered space $X$ of dimension $n$ and a commutative ring $R$.

\subsection{Definitions}\label{subsec:suppcompact}

Let $N_{*}(\Delta)$ 
and  $N^*(\Delta)$ be the  simplicial chain and cochain complexes
of  a Euclidean simplex $\Delta$, with coefficients in $R$. 
For each simplex $F \in N_{*}(\Delta)$, we write $\1_{F}$ the element of $N^*(\Delta)$ taking the value 1 on $F$ and 0 otherwise. 
Given a face $F$  of $\Delta$, we denote by $(F,0)$ the same face viewed as face of the cone $\tc\Delta=\Delta\ast [\tv]$ and by $(F,1)$ 
the face $\tc F$ of $\tc \Delta$. 
The apex is denoted $(\emptyset,1)=\tc \emptyset =[\tv]$.
Cochains on the cone are denoted $\1_{(F,\varepsilon)}$ for $\varepsilon =0$ or $1$.
For defining the blown-up intersection complex, we first set
$$\tN^*(\Delta)=N^*(\tc\Delta_0)\otimes\dots\otimes N^*(\tc\Delta_{n-1})\otimes N^*(\Delta_n).$$
A basis of $\tN^*(\Delta)$ is composed of the elements 
$\1_{(F,\varepsilon)}=\1_{(F_{0},\varepsilon_{0})}\otimes\dots\otimes \1_{(F_{n-1},\varepsilon_{n-1})}\otimes \1_{F_{n}}$,
 where 
$\varepsilon_{i}\in\{0,1\}$ and
$F_{i}$ is a face of $\Delta_{i}$ for $i\in\{0,\dots,n\}$ or the empty set with $\varepsilon_{i}=1$ if $i<n$.
We set
$|\1_{(F,\varepsilon)}|_{>s}=\sum_{i>s}(\dim F_{i}+\varepsilon_{i})$,
with the convention $\dim \emptyset=-1$.

%%%%%%%%%%%%%%%%%%%%%%%
%%%%%%%%%%%%%%%%
\begin{definition}\label{def:degrepervers}
Let $\ell\in \{1,\ldots,n\}$ and
$\1_{(F,\varepsilon)}
\in\tN^*(\Delta)$.
The  \emph{$\ell$-perverse degree} of 
$\1_{(F,\varepsilon)}\in N^*(\Delta)$ is
$$
\|\1_{(F,\varepsilon)}\|_{\ell}=\left\{
\begin{array}{ccl}
-\infty&\text{if}
&
\varepsilon_{n-\ell}=1,\\
|\1_{(F,\varepsilon)}|_{> n-\ell}
&\text{if}&
\varepsilon_{n-\ell}=0.
\end{array}\right.$$
For a cochain $\omega = \sum_b\lambda_b \ \1_{(F_b,\varepsilon_b) }\in\tN^*(\Delta)$ with 
$\lambda_{b}\neq 0$ for all $b$,
the \emph{$\ell$-perverse degree} is
$$\|\omega \|_{\ell}=\max_{b}\|\1_{(F_b,\varepsilon_b)}\|_{\ell}.$$
By convention, we set $\|0\|_{\ell}=-\infty$.
\end{definition}

%%%%%%%%%%%%%%%%%%%%%%%%%%%%%%%%%%%%%%%%
\begin{example}\label{exan:petitdessin}
We give an illustration of the perverse degree  in the 
particular case of the Euclidean 2-simplex 
$\Delta=[e_{0},e_{1},e_{2},e_{3}]$
filtered as 
$\Delta =\Delta_0 * \Delta_1 * \Delta_2$,
with $\Delta_{0}=[e_{0}]$,
$\Delta_{1}=[e_{1}]$
and
$\Delta_{2}=[e_{2},e_{3}]$.
%\bigskip

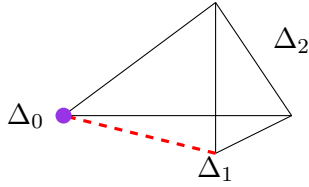
\begin{center}
\begin{tikzpicture}
\definecolor{zzttqq}{rgb}{0.6,0.2,0.9}
\draw[color=black] (2.1,1.1); 
\draw [color=black] (0,0)-- (1,0.5);
\draw [color=black] (1,0.5)--  (0,2);
\draw [color=black]  (0,2)-- (0,0);
\draw [color=black]  (1,0.5)-- (-2,0.5);
\draw [color=red,very thick,dashed]  (0,0)-- (-2,0.5);
\draw [color=black]  (0,2)-- (-2,0.5);
\fill [color=zzttqq] (-2,0.5) circle (3pt);
\draw[color=black] (-2.5,0.5) node {$\Delta_0$};
\draw[color=black] (0,-0.2) node {$\Delta_1$};
\draw[color=black] (1,1.5) node {$\Delta_2$};
\end{tikzpicture}
\end{center}
The application of \defref{def:degrepervers} gives, for instance, for $\varepsilon_{1}=1$ or 0,
$$\|\1_{ (\Delta_0,0) }\otimes \1_{(\Delta_1,\varepsilon_1)}  \otimes \1_{\Delta_2}
\|_2 =\dim \Delta_1 + \varepsilon_1+\dim \Delta_2 =1+\varepsilon_{1}.
$$
More generally, let $i\in\{0,1,2\}$ and $F_{i}$ be a face of $\Delta_{i}$, including the option
$F_{i}=\emptyset$ if the associated $\varepsilon_{i}$ is equal to 1. 
The corresponding perverse degrees may be listed as,
$$\begin{array}{ll}
\|\1_{(F_{0},1)}\otimes \1_{(F_{1},\varepsilon_{1})}\otimes \1_{F_{2}}\|_{2}
=-\infty,
&
\|\1_{(F_{0},0)}\otimes \1_{(F_{1},\varepsilon_{1})}\otimes \1_{F_{2}}\|_{2}
=
\dim (F_{1},\varepsilon_{1})+\dim F_{2},\\
\|\1_{(F_{0},\varepsilon_{0})}\otimes \1_{(F_{1},1)}\otimes \1_{F_{2}}\|_{1}
=
-\infty,
&
\|\1_{(F_{0},\varepsilon_{0})}\otimes \1_{(F_{1},0)}\otimes \1_{F_{2}}\|_{1}
=
\dim F_{2}.
\end{array}
$$
\end{example}

%%%%%%%%%%%%%%%%%%%%%%%%%%%%%%%%%%%%

These ``local'' notions are now clustered to fit in the ``global'' data of a filtered space. 

\medskip
Let $\sigma\colon \Delta=\Delta_0\ast\dots\ast\Delta_n\to X$ be a filtered simplex.
We set $\tN^*_{\sigma}=\tN^*(\Delta)$.
If $\delta_{\ell}\colon \Delta' 
\to\Delta$ 
 is  an inclusion of a face of codimension~1,
  we denote by
$\partial_{\ell}\sigma$ the filtered simplex defined by
$\partial_{\ell}\sigma=\sigma\circ\delta_{\ell}\colon 
\Delta'\to X$.
If $\Delta=\Delta_{0}\ast\dots\ast\Delta_{n}$ is filtered, the induced filtration on $\Delta'$ is denoted
$\Delta'=\Delta'_{0}\ast\dots\ast\Delta'_{n}$.
The \emph{blown-up intersection complex} of $X$ is the cochain complex 
$\tN^*(X)$ composed of the elements $\omega$
associating to each regular filtered simplex
 $\sigma\colon \Delta_{0}\ast\dots\ast\Delta_{n}\to X$
an element
 $\omega_{\sigma}\in \tN^*_{\sigma}$  
such that $\delta_{\ell}^*(\omega_{\sigma})=\omega_{\partial_{\ell}\sigma}$,
for any face operator
 $\delta_{\ell}\colon\Delta'\to\Delta$
 with $\Delta'_{n}\neq\emptyset$. 
 The differential $\delta \omega$ is defined by
 $(\delta \omega)_{\sigma}=\delta(\omega_{\sigma})$.
 The \emph{perverse degree of $\omega$ along a singular stratum $S$} equals
 $$\|\omega\|_S=\sup\left\{
 \|\omega_{\sigma}\|_{\codim S}\mid \sigma\colon \Delta\to X \;
 \text{regular such that }
 \sigma(\Delta)\cap S\neq\emptyset
 \right\}.$$
 We denote $\|\omega\|$ the map which associates $\|\omega\|_S$
 to any singular stratum $S$ and~0 to any regular one.
 A \emph{cochain $\omega\in\tN^*(X)$ is $\ov{p}$-allowable} if $\| \omega\|\leq \ov{p}$ 
 and of \emph{$\ov{p}$-intersection} if $\omega$ and $\delta\omega$ are $\ov{p}$-allowable. 
 We denote $\tN^*_{\ov{p}}(X;R)$ the complex of $\ov{p}$-intersection cochains and  
 $\crH_{\ov{p}}^*({X};R)$ its homology called
 \emph{blown-up intersection cohomology}  of $X$ 
 for the perversity~$\ov{p}$. 

%%%%%%%%%%%%%%%%%%%%%%
 \begin{definition}\label{def:support}
 Let $(X,\ov{p})$ be a perverse space.
 A non-empty subspace $K$ is a \emph{support of the cochain}
 $\omega\in \tN^*(X;R)$
if $\omega_{\sigma}=0$,
 for any regular simplex $\sigma$ such that $\sigma(\Delta)\cap K=\emptyset$.
 A cochain $\omega\in\tN^*(X;R)$ is with \emph{compact supports} if it has a compact support.
 We denote
$\tN^{*}_{\ov{p},c}(X;R)$
the complex of  $\ov{p}$-intersection cochains with compact supports and 
$\crH^{*}_{\ov{p},c}( X;R)$ its homology.
\end{definition}

\emph{When the space $X$ is compact,} we clearly have
$\crH^{*}_{\ov{p},c}(X;R)\cong \crH^{*}_{\ov{p}}(X;R)$.
As in the classical case of a manifold (see \cite[Appendix~A]{MR0440554}) the cohomology $\crH^{*}_{\ov{p},c}( X;R)$
can be obtained as a direct limit. 
To state it, we need  to recall the notion of $\cU$-small cochains in intersection cohomology.

\begin{definition}\label{def:Upetit}
Let $\cU$  be an open cover of a space $X$. An \emph{$\cU$-small simplex} is a regular simplex $\sigma\colon \Delta=\Delta_{0}\ast\dots\ast\Delta_{n}\to X$
such that there exists $U\in\cU$ with $\im\sigma\subset U$. The set of $\cU$-small simplices is denoted $\si_{\cU}$.

 The \emph{complex of blown-up $\cU$-small cochains, with coefficients in  $R$,} $\tN^{*,\cU}(X;R)$,
 is the cochain complex composed of the elements $\omega$, associating to any $\cU$-small simplex,
 $\sigma\colon\Delta= \Delta_{0}\ast\dots\ast\Delta_{n}\to X$,
 an element
 $\omega_{\sigma}\in \tN^*(\Delta)$, 
 such that $\delta_{\ell}^*(\omega_{\sigma})=\omega_{\partial_{\ell}\sigma}$,
for any face operator,
 $\delta_{\ell}\colon \Delta'_{0}\ast\dots\ast\Delta'_{n}\to \Delta_{0}\ast\dots\ast\Delta_{n}$, 
 with $\Delta'_{n}\neq\emptyset$. 
 If $\ov{p}$ is a perversity on $X$, we denote $\tN^{*,\cU}_{\ov{p}}(X;R)$ the cochain subcomplex 
of elements $\omega\in \tN^{*,\cU}(X;R)$ such that
$ \|\omega\|\leq \ov{p}$ and
$\|\delta\omega\|\leq\ov{p}$.

The set of  $\cU$-small cochains admitting a compact support is denoted
 $\tN^{*,\cU}_{c}(X)$.
 Its subcomplex  composed of the cochains of $\ov{p}$-intersection 
 is designed by $\tN^{*,\cU}_{\ov{p},c}(X;R)$
 of homology $\crH^{*,\cU}_{\ov{p},c}( X;R)$.
\end{definition}

\begin{proposition}{\cite[Theorem~B]{CST5}}\label{prop:Upetits}
Let $(X,\ov{p})$ be a perverse space and $\cU$ an open cover of $X$.
Then the restriction map,
$\rho_{\cU}\colon\tN^{*}_{\ov{p}}(X;R)\to \tN^{*,\cU}_{\ov{p}}(X;R)$,
is a  quasi-isomorphism.
\end{proposition}
 
 We establish the version with compact supports of \propref{prop:Upetits}.
 
 \begin{proposition}\label{prop:Upetitscompact}
 Let $(X,\ov{p})$ be a perverse space and $\cU$ an open cover of $X$.
Then the restriction map,
$\rho_{\cU,c}\colon\tN^{*}_{\ov{p},c}(X;R)\to \tN^{*,\cU}_{\ov{p},c}(X;R)$,
is a  quasi-isomorphism.
 \end{proposition}
 
  We postpone for a while the proof of this result.
  Recall that an open cover $\cV$ of $X$ is \emph{finer} than the open cover $\cU$ of $X$ if any element 
$V\in\cV$ is included in an element  $U\in\cU$. We denote $\cU \preceq \cV$ this relation.
If $\cU \preceq \cV$, we have an inclusion $\si_{\cV}\subset \si_{\cU}$ and a natural map
$I_{X}^{\cU,\cV}\colon \tN^{*,\cU}_{\ov{p},c}(X;R)\to \tN^{*,\cV}_{\ov{p},c}(X;R)$.
We  consider the direct limit of these maps and set
\begin{equation}\label{equa:compactlimit}
\widetilde{\N}^*_{\ov{p},c}(X;R)=\varinjlim_{\cU} \tN^{*,\cU}_{\ov{p},c}(X;R).
\end{equation}
\propref{prop:Upetitscompact} implies immediately the next characterization of $\crH^{*}_{\ov{p},c}( X;R)$.
 
 \begin{corollary}\label{cor:toutpetit}
Let $(X,\ov{p})$ be a perverse space.
The canonical map from $\tN^{*}_{\ov{p},c}(X;R)$ to the previous limit,
 $$\iota_{c}\colon \tN^{*}_{\ov{p},c}(X;R)\xrightarrow[]{\simeq} 
 \widetilde{\N}^*_{\ov{p},c}(X;R),$$
 is a quasi-isomorphism.
\end{corollary}

 %%%%%%%%%%%%%%%%%
\begin{proof}[Proof of \propref{prop:Upetitscompact}]
This is an adaptation of the proof of \cite[Theorem~B]{CST5}.
In this reference, to establish that  the map
$\rho_{\cU}\colon \tN^*_{\ov{p}}(X)\to \tN^{*,\cU}_{\ov{p}}(X)$
induces an isomorphism in homology, we introduced a cochain map,
$\varphi_{\cU}\colon \tN^{*,\cU}_{\ov{p}}(X)\to \tN^*_{\ov{p}}(X)$, and a homotopy
$\Theta\colon \tN^*_{\ov{p}}(X)\to \tN^{*-1}_{\ov{p}}(X)$
such that
$\rho_{\cU}\circ\varphi_{\cU}=\id$ and
$\delta\circ\Theta+\Theta\circ\delta=\id-\varphi_{\cU}\circ\rho_{\cU}$.
The maps $\varphi_{\cU}$ and $\Theta$ are defined from an application
$\tT\colon \tN^*(X)\to
\tN^{*-1}(X)$.
Here, it is thus sufficient to prove that  the image by these maps of a cochain with compact supports 
is a cochain with the same support. This is direct for $\rho_{\cU}$.

Concerning $\tT$, we consider $\omega\in \tN^{*}_{\ov{p},c}(X)$. 
By definition, there exists a compact $L\subset X$ such that $\omega_{\sigma}=0$ for any regular simplex
 $\sigma\colon \Delta\to X$
 such that $\sigma(\Delta)\cap L=\emptyset$.
By definition of $\tT$ (see \cite[Proposition 9.8]{CST5}), we have $(\tT(\omega))_{\sigma}=\tT_{\Delta}(\omega_{K(\Delta)})$,
with
 $$\omega_{K(\Delta)}=
\sum_{\substack{F\ast G\subset K(\Delta) \\  |(F\ast G,\varepsilon)|=k\phantom{-}}}
\omega_{\sigma_{_{F\ast G}}}(F\ast G,\varepsilon) \,\1_{(F\ast G,\varepsilon)},
$$
where the simplex $\sigma_{_{F\ast G}}$is a restriction of $\sigma$.
Therefore, the image of $\sigma_{_{F\ast G}}$ is included in the image of $\sigma$
and we have $\omega_{K(\Delta)}=0$ as required. 
 \end{proof}
 
  We introduce now the natural isomorphisms in the category of perverse spaces. 
 By forgetting the information on perversities, they are the stratified homeomorphisms
  of  \cite[Definition 1.5]{CST5}.
 
  \begin{definition}
 A \emph{stratified homeomorphism between perverse spaces} $f\colon (X,\ov{p})\to (Y,\ov{p}')$, 
 is a homeomorphism $f\colon X\to Y$
 between filtered spaces of the same dimension,
 such that $f(X_{i})=Y_{i}$ and $\ov{p}(S)=\ov{p}'(f(S))$ for any stratum $S\subset X$.
 \end{definition}
 
The next property is a consequence of \cite[Theorem A]{CST5}.
 
 \begin{proposition}\label{prop:homeo}
 Let $f\colon (X,\ov{p})\to (Y,\ov{p}')$ be a stratified homeomorphism between perverse spaces. Then 
 $f$ induces  a chain isomorphism,
 $\tN^{*}_{\ov{p}',c}(Y;R)\cong \tN^{*}_{\ov{p},c}(X;R)$,
 and thus an isomorphism $\crH^{*}_{\ov{p}',c}(Y;R)\cong \crH^{*}_{\ov{p},c}(X;R)$.
 \end{proposition}
 
%%%%%%%%%%%%%%%%%%%%%%%%%%%%%%%%%%%%%%%
\subsection{Cup and cap products}\label{sec:cupcap}
We have already defined a cup product in \cite{CST1}  and ${\text{cup}_i}$ products
in \cite{CST2} on the blown-up intersection cochains in the case of \ffss~with GM-perversities. 
In \cite{CST5}, a definition of a cup product has been made in the general case considered here; we recall this definition.

\begin{definition}\label{def:cupsurDelta}
Two ordered simplices
 $F=[a_{0},\dots,a_{k}]$ and $G=[b_{0},\dots,b_{\ell}]$
 of a Euclidean simplex $\Delta$ are  \emph{compatible}
 if $a_{k}=b_{0}$.
 In this case, we set
 $$F\cup G=[a_{0},\dots,a_{k},b_{1},\dots,b_{\ell}]\in N_{*}(\Delta).$$
 The \emph{cup product on $N^*(\Delta)$} is defined on the dual basis by
 $$\1_{F}\cup \1_{G}=(-1)^{|F|\,|G|} \1_{F\cup G},$$ 
 if $F$ and $G$ are compatible and~0 otherwise. 
 If $\Delta=\Delta_{0}\ast\dots\ast\Delta_{n}$ is a regular Euclidean simplex, 
this product is extended to $\tN^*(\Delta)$ with the classical rule of commutation of graded objects, as follows:\\
 if
 $\omega_{0}\otimes\dots\otimes\omega_{n}$ and 
 $\eta_{0}\otimes\dots\otimes\eta_{n}$
are elements of 
 $N^*(\tc\Delta_{0})\otimes\dots\otimes N^*(\Delta_{n})$,
we set
 \begin{equation}\label{equa:cupsimplexefiltre}
 (\omega_{0}\otimes\dots\otimes\omega_{n})\cup
 (\eta_{0}\otimes\dots\otimes\eta_{n})=
 (-1)^{\sum_{i>j}|\omega_{i}|\,|\eta_{j}|}
 (\omega_{0}\cup\eta_{0})\otimes\dots\otimes
 (\omega_{n}\cup\eta_{n}).
 \end{equation}
 \end{definition}
 Recall the main property of this cup product.
 
 \begin{proposition}{\cite[Proposition 4.2]{CST5}}\label{prop:cupprduitespacefiltre}
 Let $X$ be a filtered space endowed with two perversities $\ov{p}$ and $\ov{q}$.
 The previous cup product gives an associative product,
 \begin{equation}\label{equa:cupprduitespacefiltre}
 -\cup -\colon\tN^k_{\ov{p}}(X;R)\otimes \tN^{\ell}_{\ov{q}}(X;R)\to \tN^{k+\ell}_{\ov{p}+\ov{q}}(X;R),
 \end{equation}
defined by $(\omega\cup\eta)_{\sigma}=\omega_{\sigma}\cup \eta_{\sigma}$,
 for any $(\omega,\eta)\in \tN^*_{\ov{p}}(X;R)\times \tN^*_{\ov{q}}(X;R)$
 and any regular filtered simplex
 $\sigma\colon\Delta\to X$.
 Moreover, this morphism induces a graded commutative product, called intersection cup product,
  \begin{equation}\label{equa:cupprduitTWcohomologie}
-\cup -\colon \crH^k_{\ov{p}}(X;R)\otimes \crH^{\ell}_{\ov{q}}(X;R)\to
 \crH^{k+\ell}_{\ov{p}+\ov{q}}(X;R).
 \end{equation}
 \end{proposition}

We recall  the intersection cap product studied in \cite{CST5} and \cite{CST7}.
Let $\Delta=[e_{0},\dots,e_{r},\dots,e_{m}]$ be a Euclidean simplex.  
The (classical) cap product 
 $-\cap \Delta\colon N^*(\Delta)\to N_{m-*}(\Delta)$ is defined by
$$
\1_{F}\cap \Delta=\left\{
\begin{array}{cl}
[e_{r},\dots,e_{m}] 
&
\text{if } F=[e_{0},\dots,e_{r}], \text{ for any } r\in\{0,\dots,m\},\\
0
&
\text{otherwise.}
\end{array}\right.
$$
If $\Delta=\Delta_{0}\ast\dots\ast\Delta_{n}$ is a regular filtered simplex, we set
$\tDelta=\rc\Delta_0\times\dots\times\rc \Delta_{n-1}\times\Delta_n$.
The previous cap product is extended with the rule of permutation of graded objects as follows.
If
$\1_{(F,\varepsilon)}=\1_{(F_{0},\varepsilon_{0})}\otimes\dots\otimes \1_{(F_{n-1},\varepsilon_{n-1})}\otimes \1_{F_{n}}
\in \tN^*(\Delta)$, we define
\begin{eqnarray}
\1_{(F,\varepsilon)}\,\widetilde{\cap}\, \tDelta
&=&
(-1)^{\nu(F,\varepsilon,\Delta)}
(\1_{(F_{0},\varepsilon_{0})}\cap \tc\Delta_{0})\otimes\dots\otimes
(\1_{(F_{n-1},\varepsilon_{n-1})}\cap \tc\Delta_{n-1})\otimes
(\1_{F_{n}}\cap \Delta_{n}),\nonumber\\
&\in&
N_{*}(\tc\Delta_{0})\otimes\dots\otimes N_{*}(\tc \Delta_{n-1})\otimes N_{*}(\Delta_{n}),\label{equa:lecapsurdelta}
\end{eqnarray}
where $\nu(F,\varepsilon,\Delta)=\sum_{j=0}^{n-1}(\dim\Delta_{j}+1) \,(\sum_{i=j+1}^n|(F_{i},\varepsilon_{i})|)
$, with the convention $\varepsilon_{n}=0$.

\medskip
An intersection cap product on $\tDelta$ must take values in the chain complex $N_{*}(\Delta)$. For that, we construct 
a morphism
$\mu_{\Delta}\colon N_{*}(\tc\Delta_{0})\otimes\dots\otimes N_{*}(\tc \Delta_{n-1})\otimes N_{*}(\Delta_{n})
\to N_{*}(\Delta)$, by its value on
$(F,\varepsilon)=
(F_{0},\varepsilon_{0})\otimes\dots\otimes (F_{n-1},\varepsilon_{n-1})\otimes F_{n}$.
Let $\ell$ be the smallest integer $j$ such that $\varepsilon_{j}=0$. We set
\begin{equation}\label{equa:eclatement}
\mu_{\Delta}(F,\varepsilon)=\left\{
\begin{array}{cl}
F_{0}\ast\dots\ast F_{\ell}&\text{if } 
\dim (F,\varepsilon) =\dim (F_{0}\ast\dots\ast F_{\ell}),\\
0&\text{otherwise.}
\end{array}\right.
\end{equation}
This application is a chain map (cf. \cite[Lemma 6.5]{CST5}) and we define the  \emph{local intersection cap product} 
$$-\cap \Delta\colon \tN^*(\Delta)\to N_{m-*}(\Delta) \text{  as  }
\omega\cap \Delta=\mu_{\Delta}(\omega\,\widetilde{\cap}\,\tDelta).$$
This expression may be carried on to filtered simplices of a filtered space $X$.

\begin{definition}\label{def:capglobal}
Let $\omega\in\tN^*(X;R)$ and  $\sigma\colon \Delta_{\sigma}\to X$ be a filtered simplex. We set
$$\omega\cap \sigma=\left\{
\begin{array}{cl}
\sigma_{*}(\omega_{\sigma}\cap \Delta_{\sigma}) 
&
\text{if } \sigma \text{ is regular,}\\
0&
\text{otherwise.}
\end{array}\right.$$
With a linear extension, the \emph{intersection cap product} is defined as a map
$$-\cap - \colon \tN^k(X;R)\otimes C_{m}(X;R) \to C_{m-k}(X;R).$$
\end{definition}
As proved in \cite{CST5}, the cap product respects the tame intersection chains.

\begin{proposition}{\cite[Propositions 6.6 and 6.7]{CST5}}\label{prop:lecap}
Let $X$ be a filtered space endowed with two perversities
$\ov{p}$ and $\ov{q}$. 
The cap product defines a homomorphism
$$-\cap - \colon \tN^k_{\ov{p}}(X;R)\otimes \gC_{m}^{\ov{q}}(X;R) \to \gC_{m-k}^{\ov{p}+\ov{q}}(X;R)$$
satisfying the following properties. 
\begin{enumerate}[\rm (i)]
\item This is a chain map:  
$\gd (\omega\cap\xi)=(\delta \omega)\cap \xi+(-1)^{|\omega|}\omega\cap(\gd \xi)$.
\item The cap and the cup products are compatible:  
$(\omega\cup\eta)\cap \xi=
(-1)^{|\omega|\,|\eta|} 
\eta\cap(\omega\cap\xi)$.
\end{enumerate}
\end{proposition}

%%%%%%%%%%%%%%%%%%%%%%%%%%%%%%%%%%%%%
\subsection{Properties of the blown-up cohomology with compact supports}\label{sec:propcompact}
In this section, we establish the properties allowing the use of 
\propref{prop:superbredon} for the proof of \thmref{thm:invariance} and \thmref{thm:dual}. 
In particular, we construct a Mayer-Vietoris exact sequence, compute the 
intersection cohomology of a cone and of a
product with $\R$, in the case of compact supports.

\medskip
\paragraph{\bf Cochains with compact supports on an open subset}

Let $U$ be an open subset of a filtered space $X$ and $\cV$ an open cover of $U$.
To any
$\omega\in \tN^{*,\cV}_{c}(U;R)$ of compact support $K\subset U$, we associate the open cover
$\cU=\cV\cup\{X\backslash K\}$ of $X$. 
Let $\sigma\colon\Delta\to X$, 
$\sigma\in\si_{\cU}$, be a regular simplex. We define:
\begin{equation}\label{equa:eta}
\eta_{\sigma}=\left\{
\begin{array}{ccl}
\omega_{\sigma}
&\text{if}&
\sigma\in\si_{\cV},\\
0
&\text{if}&
\im \sigma\cap K=\emptyset.
\end{array}\right.
\end{equation}
There is no ambiguity in this construction since $K$ is a support of $\omega$. 

\medskip
Let $\delta_{\ell}\colon \Delta'\to\Delta$ be a regular face of codimension~1. The conditions
$\sigma\in\si_{\cV}$ and $\im\sigma\cap K=\emptyset$ imply
$\partial_{\ell}\sigma\in\si_{\cV}$ and $\im\partial_{\ell}\sigma\cap K=\emptyset$,
respectively. 
Therefore, the compatibility of $\omega$ with the face operators
gives 
$\delta_{\ell}^*\eta_{\sigma}=\eta_{\partial_{\ell}\sigma}$. 
Moreover, as $K\subset U\subset X$ is a compact support of $\eta$, we get  $\eta\in\tN^{*,\cU}_{c}(X;R)$. 
Let $ \pmb{\eta}$ be the class of $\eta$ in $\widetilde{\N}^*_{\ov{p},c}(X;R)$, 
see (\ref{equa:compactlimit}).
The association $\omega\mapsto \pmb{\eta}$ defines an application
$$I_{U,X}^{\cV}\colon \tN^{*,\cV}_{c}(U;R) \to \widetilde{\N}^{*}_{c}(X;R),$$
compatible with the differentials since $(\delta\omega)_{\sigma}=\delta(\omega_{\sigma})$.

\medskip
Let $\ov{p}$ be a perversity on $X$. We endow the open subset $U\subset X$ with the
induced perversity also denoted $\ov{p}$. 
Let $\omega\in \tN^{*,\cV}_{\ov{p},c}(U)$ of compact support $K$. 
For any regular simplex $\sigma\in\si_{\cU}$ and any $\ell\in\{1,\dots,n\}$, we have
the inequality
 $\|\eta_{\sigma}\|_{\ell}\leq \|\omega_{\sigma}\|_{\ell}$ from which we deduce a cochain map,
$$I_{U,X}^{\cV}\colon \tN^{*,\cV}_{\ov{p},c}(U;R) \to \widetilde{\N}^{*}_{\ov{p},c}(X;R).$$

\begin{proposition}\label{prop:compactouvertdeX}
Let $(X,\ov{p})$ be a perverse space and  $U$  an open subspace, 
endowed with the induced perversity. 
The maps 
$I_{U,X}^{\cV}$ defined above induce an injective application of cochain complexes,
$$\pmb{I}_{U,X}\colon \widetilde{\N}^*_{\ov{p},c}(U;R)\to \widetilde{\N}^*_{\ov{p},c}(X;R).$$
\end{proposition}

\begin{proof}
Let $\cV\preceq \cV'$ be two open covers of $U$. The open covers
$\cU=\cV\cup\{X\backslash K\}$
and
$\cU'=\cV'\cup\{X\backslash K\}$
of $X$ satisfy $\cU\preceq \cU'$. 
Thus we have a commutative diagram,
$$\xymatrix{
\tN^{*,\cV}_{\ov{p},c}(U)
\ar[r]^-{I_{U,X}^{\cV}}
\ar[d]_{I_{U}^{\cV,\cV'}}
&
\widetilde{\N}^{*}_{\ov{p},c}(X).
\\
\tN^{*,\cV'}_{\ov{p},c}(U)
\ar[ru]_-{I_{U,X}^{\cV'}}
&
}$$ 
The map $\pmb{I}_{U,X}$ is then obtained by a passage to the limit.
For proving the injectivity, we consider $\pmb{\omega}\in \widetilde{\N}^*_{\ov{p},c}(U)$ such that
$\pmb{I}_{U,X}(\pmb{\omega})=0$.
The class
$\pmb{\omega}$ is the limit of elements  $\omega\in\tN^{*,\cV}_{\ov{p},c}(U)$, 
where $\cV$ is an open cover of  $U$.
Let  $\eta\in \tN^{*,\cU}_{\ov{p},c}(X)$ be the element associated to $\omega$
in (\ref{equa:eta}).
From $\pmb{I}_{U,X}(\pmb{\omega})=0$, we get the existence of an open cover $\cW$ 
of  $X$ finer than $\cU$ and such that
$\eta_{\sigma}=0$ for any $\sigma\in\si_{\cW}$.
Thus the open cover $\cW\cap \cV$ of $U$ 
formed of the intersections of elements of  $\cV$ and $\cW$
verifies by definition
$\eta_{\sigma}=0$ 
for any $\sigma\in \si_{\cW\cap \cV}=\si_{\cW}\cap \si_{\cV}$. It follows $\pmb{\omega}=0$.
\end{proof}

%%%%%%%%%%%%%%%%%%

 \paragraph{\bf Mayer-Vietoris exact sequence with compact supports}

\begin{proposition}\label{prop:MVcompact}
Let $(X,\ov{p})$ be a locally compact and paracompact perverse space. % 
The induced perversities on the open subsets of $X$ are also denoted $\ov{p}$.
If $X=U_{1}\cup U_{2}$ is an open cover of $X$, then the sequence,
$$
\xymatrix@1{
0\ar[r]&
\widetilde{\N}^*_{\ov{p},c}(U_{1}\cap U_{2};R)
\ar[r]^-{(\pmb{I}_{1},\pmb{I}_{2})}
&
\widetilde{\N}^*_{\ov{p},c}(U_{1};R)\oplus \widetilde{\N}^*_{\ov{p},c}(U_{2};R)
\ar[r]^-{\pmb{I}_{3}-\pmb{I}_{4}}
&
\widetilde{\N}^*_{\ov{p},c}(X;R)\ar[r]
&
0,
}$$
whose applications $\pmb{I}_{\bullet}$ are defined in \propref{prop:compactouvertdeX},
is exact.
\end{proposition}

Before giving the proof, we recall the following result from \cite{CST5}.

\begin{lemma}{\cite[Lemma 10.2]{CST5}}\label{lem:0cochaine}
Let $(X,\ov{p})$ be a perverse space. Each application $g\colon X\to R$
defines a  0-cochain
$\tg\in \tN^0_{\ov{0}}(X)$.
Moreover the association $g\mapsto \tilde{g}$ is $R$-linear.
\end{lemma}

The cochain $\tilde{g}$ is defined as follows. Let
$\sigma\colon \Delta_0\ast\dots\ast\Delta_n\to X$
be a regular filtered simplex and
$b=(b_0,\dots,b_n)\in \tc \Delta_0\times\dots\times\tc \Delta_{n-1}\times \Delta_n$. 
We set 
$i_0=\min\left\{i\mid b_i\in \Delta_i\right\}$
and
$\tilde{g}_{\sigma}(b)=g(\sigma(b_{i_0}))$.

\begin{proof}[Proof of \propref{prop:MVcompact}]
The injectivity of $(\pmb{I}_{1},\pmb{I}_{2})$ is a consequence of \propref{prop:compactouvertdeX}. 
The rest of the proof goes along the next steps.

\smallskip
$\bullet$ \emph{The map $(\pmb{I}_{3}-\pmb{I}_{4})\circ (\pmb{I}_{1},\pmb{I}_{2})$
 is constant with value~0.}\\
Let $\pmb{\omega}\in \widetilde{\N}^*_{\ov{p},c}(U_{1}\cap U_{2})$.
Consider an open cover  $\cV$ of $U_{1}\cap U_{2}$ and a cochain 
$\omega\in \tN^{*,\cV}_{\ov{p},c}(U_{1}\cap U_{2})$ with compact support 
$K\subset U_{1}\cap U_{2}\subset X$, representing $\pmb{\omega}$.
We set $\pmb{\eta}_{1}=\pmb{I}_{3}(\pmb{I}_{1}(\omega))$,
$\pmb{\eta}_{2}=\pmb{I}_{4}(\pmb{I}_{2}(\omega))$
and choose representing elements of $\pmb{\eta}_{i}$, for $i=1,\,2$, 
$$\eta_{i}\in \tN^{*,\cV\cup\{U_{i}\backslash K\}\cup \{X\backslash K\}}_{\ov{p},c}(X)=
\tN^{*,\cV\cup \{X\backslash K\}}_{\ov{p},c}(X).$$
From the definition of the applications $\pmb{I}_{\bullet}$, we have for $i=1,\,2$ and $\sigma\in \si_{\cV\cup\{X\backslash K\}}$,
$$(\eta_{i})_{\sigma}=\left\{
\begin{array}{cl}
\omega_{\sigma}
&
\text{if}\quad \sigma\in\si_{\cV},\\
0
&
\text{if}\quad \im\sigma\cap K=\emptyset.
\end{array}\right.
$$
This implies
$(\eta_{1})_{\sigma}=(\eta_{2})_{\sigma}$
and $\pmb{I}_{3}\circ \pmb{I}_{1}=\pmb{I}_{4}\circ \pmb{I}_{2}$.

\smallskip
$\bullet$ \emph{The kernel of $\pmb{I}_{3}-\pmb{I}_{4}$ is included in the image of 
$(\pmb{I}_{1},\pmb{I}_{2})$.}\\
Let $\pmb{\omega}_{i}\in \widetilde{\N}^{*}_{\ov{p},c}(U_{i})$, for $i=1,\,2$, such that 
$\pmb{I}_{3}(\pmb{\omega}_{1})=\pmb{I}_{4}(\pmb{\omega}_{2})$.
Consider an open cover $\cV_{i}$
of $U_{i}$ and a cochain $\omega_{i}\in \tN^{*,\cV_{i}}_{\ov{p},c}(U_{i})$
with compact support $K_{i}\subset U_{i}\subset X$, representing $\pmb{\omega}_{i}$
for $i=1,\,2$.
With the local compacity, there exist an open subset $W$ and a compact subset $F$ such that
$K_{1}\cap K_{2}\subset W\subset F\subset U_{1}\cap U_{2}$. 
Set $\pmb{I}_{3}(\pmb{\omega}_{1})=\pmb{\eta}_{1}$ and 
$\pmb{I}_{4}(\pmb{\omega}_{2})=\pmb{\eta}_{2}$.
From the definition of the applications  $\pmb{I}_{\bullet}$,
we get $\eta_{i}\in \tN^{*,\cV_{i}\cup\{X\backslash K_{i}\}}_{\ov{p},c}(X)$ and
\begin{eqnarray}
(\eta_{i})_{\sigma}&=&(\omega_{i})_{\sigma}
\quad\text{if} \quad\sigma\in\si_{\cV_{i}},\label{equa:etaomega}\\
&=&0
\quad\text{if} \quad\im\sigma\cap K_{i}=\emptyset. \label{equa:eta0}
\end{eqnarray}
The equality $\pmb{\eta}_{1}=\pmb{\eta}_{2}$ implies the existence of an open cover $\cU$ of $X$
such that $\cV_{i}\cup\{X\backslash K_{i}\}\preceq \cU$ for $i=1,\,2$ and
\begin{equation}\label{equa:eta12}
(\eta_{1})_{\sigma}=(\eta_{2})_{\sigma}\quad\text{if}\quad \sigma\in\si_{\cU}.
\end{equation}
Without loss of generality, we may suppose
$\{X\backslash K_{1}, X\backslash K_{2},W\}\preceq \cU$. 
In particular, for any $U\in\cU$, we have
\begin{equation}\label{equa=lecU}
U\cap K_{1}=\emptyset \text{ or }
U\cap K_{2}=\emptyset \text{ or }
U\subset W.
\end{equation}
Thus the open cover $\cW=\{U\cap U_{1}\cap U_{2}\mid U\in\cU\}$ of $U_{1}\cap U_{2}$
 can be decomposed in
$\cW=\cW_{1}\cup\cW_{2}\cup\cW_{3}$
with $\cW_{i}=\{U\in\cW\mid U\cap K_{i}=\emptyset\}$ for $i=1,\,2$ and
$\cW_{3}=\{U\in\cW\mid U\subset W\}$.
For any regular simplex
$\sigma\in\si_{\cW}$, we set
\begin{equation}\label{equa:omega}
\omega_{\sigma}=
\left\{\begin{array}{ccl}
(\eta_{1})_{\sigma}=(\eta_{2})_{\sigma}
&\text{if}&
\sigma\in\si_{\cW_{3}},\\
0
&\text{if}&
\sigma\in \si_{\cW_{1}}\cup \si_{\cW_{2}}.
\end{array}\right.
\end{equation}
The following paragraphs establish the validity of that definition.
\begin{itemize} 
\item $(\eta_{1})_{\sigma}=(\eta_{2})_{\sigma}$ if $\sigma\in\si_{\cW_{3}}\subset \si_{\cU}$, cf. (\ref{equa:eta12}).
\item $(\eta_{i})_{\sigma}=0$ if $\sigma\in\si_{\cW_{i}}\cap\si_{\cW_{3}}$, for $i=1,\,2$, cf. (\ref{equa:eta0}).
\item $\delta_{\ell}\omega_{\sigma}=\omega_{\partial_{\ell}\sigma}$
for any face operator because $\eta_{1}$ satisfies this property and 
$\sigma\in \si_{{\cW_{i}}}$ implies $\partial_{\ell}\sigma\in \si_{\cW_{i}}$ for $i=1,\,2,\,3$.
\item $\|\omega_{\sigma}\|_{\ell}\leq \|(\eta_{1})_{\sigma}\|_{\ell}$
and
$\|\delta \omega_{\sigma}\|_{\ell}\leq \|(\delta\eta_{1})_{\sigma}\|_{\ell}$
for any $\sigma\in \si_{\cW_{3}}$ and $\ell\in\{1,\dots,n\}$.
\item For any $\sigma\in\si_{\cW}$, the property
$\im\sigma\cap F=\emptyset$ implies $\omega_{\sigma}=0$ because
$\sigma\in \si_{\cW_{1}}\cup \si_{\cW_{2}}$.
(Note that $F$ is a compact support of $\omega$.)
\end{itemize}
We have constructed a cochain $\pmb{\omega}\in \widetilde{\N}^{*}_{\ov{p},c}(U_{1}\cap U_{2})$
and we are reduced to prove
$\pmb{I}_{i}(\pmb{\omega})=\pmb{\omega}_{i}$.
We do it for $i=1$, the second case being similar.
Set $\pmb{\gamma}_{1}=\pmb{I}_{1}(\pmb{\omega})$. 

We set $\cW'=\{U\cap U_{1}\backslash F\mid U\in\cU\}$ and denote
$\cH=\cW\cup\cW'$ the open cover of $U_{1}$.
This cover is a refinement of $\cW\cup \{U_{1}\backslash F\}$ and it is sufficient to prove
$(\gamma_{1})_{\sigma}=(\omega_{1})_{\sigma}$
for any $\sigma\in\cH$.
The cover $\cH$ being also a refinement of $\cU$ and therefore of 
$\{X\backslash K_{1}, X\backslash K_{2},W\}$ it is sufficient to consider the three following cases.

-- If $\im\sigma\cap K_{1}=\emptyset$, by using the fact that $F$ is a compact support of $\omega$, 
we have,
$$(\gamma_{1})_{\sigma}=_{(\ref{equa:eta})}
\left\{
\begin{array}{ccl}
\omega_{\sigma}&
\text{if}&
\sigma\in\si_{\cW}\\
0&
\text{if}&
\sigma\in\si_{\cW'}
\end{array}\right.
=
\left\{
\begin{array}{ccl}
\omega_{\sigma}&
\text{if}&
\sigma\in\si_{\cW_{1}}\\
0&
\text{if}&
\sigma\in\si_{\cW'}
\end{array}\right.
=_{(\ref{equa:omega})} 0
=(\omega_{1})_{\sigma},
$$

-- If $\im\sigma\cap K_{2}=\emptyset$ and
$\im\sigma\cap K_{1}\neq\emptyset$, we have
$$(\gamma_{1})_{\sigma}=_{(\ref{equa:eta})}
\left\{
\begin{array}{ccl}
\omega_{\sigma}&
\text{if}&
\sigma\in\si_{\cW}\\
0&
\text{if}&
\sigma\in\si_{\cW'}
\end{array}\right.
=
\left\{
\begin{array}{ccl}
\omega_{\sigma}&
\text{if}&
\sigma\in\si_{\cW_{2}}\\
0&
\text{if}&
\sigma\in\si_{\cW'}
\end{array}\right.
=_{(\ref{equa:omega})} 0.
$$
As $\im\sigma\cap K_{1}\neq\emptyset$, we get $\sigma\in\si_{\cV_{1}}$ and
$$(\omega_{1})_{\sigma}
=_{(\ref{equa:etaomega})}
(\eta_{1})_{\sigma}
=_{(\ref{equa:eta12})}
(\eta_{2})_{\sigma}
=_{(\ref{equa:eta0})}
0.$$

-- If $\im\sigma\subset U\subset W$ and
$\im\sigma\cap K_{1}\neq\emptyset$,
we have $U\in \cW_{3}$ and $U\in \cV_{1}$. This implies
$$(\gamma_{1})_{\sigma}
=_{(\ref{equa:eta})}
\omega_{\sigma}
=_{(\ref{equa:omega})}
(\eta_{1})_{\sigma}
=_{(\ref{equa:etaomega})}
(\omega_{1})_{\sigma}.
$$

\smallskip
$\bullet$ \emph{The map $\pmb{I}_{3}-\pmb{I}_{4}$ is onto.}
Let $\pmb{\omega}\in \widetilde{\N}^*_{\ov{p},c}(X)$. 
Consider an open cover $\cU$ of $X$ and
a cochain $\omega\in \tN^{*,\cU}_{\ov{p},c}(X)$ with compact support $K\subset X$, 
representing $\pmb{\omega}$.

From the paracompactness of $X$, we get two fonctions, 
$g_{i}\colon X\to\{0,1\}$, $i=1,\,2$, satisfying
$\Supp g_{i}\subset U_{i}$ and $g_{1}+g_{2}=1$. 
We denote  $\tilde{g}_{i}\in \tN^0_{\ov{0}}(X)$ the associated 0-cochain, defined in
\lemref{lem:0cochaine}.
There exist also two relatively compact open subsets  $W_{i}$ such that
$\Supp g_{i}\cap K\subset W_{i}\subset\ov{W}_{i}\subset U_{i}$.
We fix $i=1$. We already know 
$\tilde{g}_{1}\cup \omega\in \tN^{*,\cU}_{\ov{p}}(X)$. 
We define 
$$\cA
=
\{ V\cap U_{1}\backslash K \mid V\in \cU\},
\;
\cB
=
\{V\cap U_{1}\backslash \Supp g_{1} \mid V\in \cU\},
\;
\cC
=
\{V\cap W_{1} \mid V\in \cU\},
$$
and consider the open cover
$\cV_{1}=\cA\cup \cB\cup \cC$
of $U_{1}$. 
By restriction, we have $\tilde{g}_{1}\cup \omega\in\tN^{*,\cV_{1}}_{\ov{p}}(U_{1})$.
If $\sigma\in\si_{\cV_{1}}$ is such that $\im\sigma\cap \ov{W}_{1}=\emptyset$, then we have
$\im\sigma\cap \Supp g_{1}=\emptyset$
or
$\im\sigma\cap K=\emptyset$. 
In each case, we may write
$(\tilde{g}_{1}\cup \omega)_{\sigma}
=0$. 
 Therefore $\ov{W}_{1}$ is a compact support of $\tilde{g}_{1}\cup \omega$
 and we get
$\tilde{g}_{1}\cup \omega\in \tN^{*,\cV_{1}}_{\ov{p},c}(U_{1})$.

We use the same process for $i=2$ with an open cover $\cV_{2}$ of $U_{2}$ 
and the cochain $\tilde{g}_{2}\cup\omega\in \tN^{*,\cV_{2}}_{\ov{p},c}(U_{2})$.
By choosing an open cover $\cX$ of $X$ finer than $\cV_{1}\cup \{X\backslash \ov{W}_{1}\}$ and 
than $\cV_{2}\cup \{X\backslash \ov{W}_{2}\}$, we see that
$\pmb{I}_{3}-\pmb{I}_{4}$ sends the class associated to
$(\tilde{g}_{1}\cup\omega,-\tilde{g}_{2}\cup\omega)$ on
$\pmb{\omega}$. This proves the surjectivity of $\pmb{I}_{3}-\pmb{I}_{4}$.
\end{proof}

%%%%%%%%%%%%%%%%%%%%%%
\paragraph{\bf Cohomology with compact supports of a cone}

For the next two propositions, we need a technical result on some intersection cochains defined on the product
with the real line.

\begin{lemma}\label{lem:LRD}
Let $(X,\ov{p})$ be a perverse space. We consider the following subcomplexes of
$\tN^*_{\ov{p}}(X\times \R;R)$,
\begin{eqnarray*}
\cL^*
&=&
\left\{\omega 
\mid
\exists a>0 \text{ such that }
\omega_{\sigma}=0 \text{ if }\im\sigma\cap(X\times ]a,\infty[)=\emptyset\right\},\\
\cR^*
&=&
\left\{\omega 
\mid
\exists b>0 \text{ such that }
\omega_{\sigma}=0 \text{ if }\im\sigma\cap(X\times ]-\infty,-b[)=\emptyset\right\},\\
\cK^*
&=&
\left\{\omega 
\mid
\exists K \text{compact, such that }
K\subset X \text{ and }
\omega_{\sigma}=0 \text{ if }\im\sigma\cap(K\times \R)=\emptyset\right\}.
\end{eqnarray*}
Then the complexes
$\cL^*$, $\cR^*$, $\cL^*\cap \cK^*$ and $\cR^*\cap \cK^*$
are acyclic.
\end{lemma}

\begin{proof}
$\bullet$ \emph{The complex $\cL^*$ is acyclic.}
Let $\omega\in\cL^k$, $\delta\omega=0$. Denote $a$ the positive number associated to  $\omega$,
$I_{0}\colon X\to X\times \R$ the map defined by
$I_{0}(x)=(x,0)$ and
$\pr\colon X\times \R\to X$  the canonical projection. 
\cite[Proposition 11.3]{CST5} gives a homotopy 
$\Theta\colon \tN^*_{\ov{p}}(X\times\R)\to \tN^{*-1}_{\ov{p}}(X)$
such that
\begin{equation}\label{equa:homotopiea}
\Theta\circ\delta+\delta\circ\Theta=(I_{0}\circ\pr)^*-\id.
\end{equation}
We are going to prove  $\delta(\Theta(\omega))=-\omega$ and
$\Theta(\omega)\in\cL_{*}$.

For any regular simplex
$\sigma\colon \Delta\to X\times \R$,  we have 
$(I_{0}\circ\pr)^*(\omega_{\sigma})=\omega_{I_{0}\circ\pr\circ\sigma}=0$
because
$\im(I_{0}\circ\pr\circ\sigma)\cap (X\times ]a,\infty[)=0$.
Thus (\ref{equa:homotopiea}) becomes
$\delta(\Theta(\omega))=-\omega$.

Let $\Delta\otimes [0,1]$ be the simplicial complex whose simplices are the joins
$F\ast G$ with $F\subset \Delta\times \{0\}$
and $G\subset \Delta\times \{1\}$.
Let $\sigma\colon \Delta\to X\times \R$, $\sigma(x)=(\sigma_{1}(x),\sigma_{2}(x))$, such that 
$\im\sigma\subset X\times ]-\infty,a]$.
At $\sigma$, we associate
$\hat{\sigma}\colon \Delta\otimes [0,1]\to X\times \R$
defined by
$\hat{\sigma}(x,t)=(\sigma_{1}(x),t\sigma_{2}(x))$.
The expression of $\Theta(\omega)_{\sigma}$ given in the proof of \cite[Proposition 11.3]{CST5} 
 depends on elements of the form
$\omega_{\hat{\sigma}\circ \iota_{F\ast G}}$ with 
$\im (\hat{\sigma}\circ \iota_{F\ast G})\subset X\times ]-\infty,a]$. This implies
$(\Theta(\omega))_{\sigma}=0$ and $\Theta(\omega)\in\cL^*$.

\medskip\noindent
$\bullet$ \emph{The complex $\cL^*\cap \cK^*$ is acyclic.} 
Let $\omega\in \cL^*\cap \cK^*$ with $\delta\omega=0$. Denote $a$  
the positive  number and $K$ the compact subset of $X$ associated to $\omega$.
With the previous notations, we observe that the condition
$(\im\sigma)\cap(K\times \R)=\emptyset$
is equivalent to
$(\pr(\im\sigma))\cap K=\emptyset$.
The previous map $\hat{\sigma}$ satisfies $\pr(\im\sigma)=\pr(\im\hat{\sigma})$.
Thus if $\sigma$ is such that $\pr(\im\sigma)=\emptyset$ implies $\omega_{\sigma}=0$,  we have also
$\omega_{\hat{\sigma}\circ\iota_{F\ast G}}=0$. We deduce $\Theta(\omega)\in\cL^*\cap\cK^*$, 
with the same number $a$ and the same compact subset $K$.

\medskip\noindent
$\bullet$ The proofs \emph{of acyclicity of $\cR^*$ and $\cR^*\cap \cK^*$} are similar.
\end{proof}

\begin{proposition}\label{prop:conecompact}
Let $X$  be a compact filtered space.
The cone $\rc X$ is endowed with the conic filtration and with a perversity $\ov{p}$. 
We denote also $\ov{p}$ the induced perversity on $X$.
Then the following properties are satisfied for any commutative ring~$R$.
\begin{enumerate}[\rm (a)]
\item For any $k\geq \ov{p}(\tw)+2$, there exists an isomorphism,
$$\crH^{k-1}_{\ov{p}}(X;R)\xrightarrow[]{\cong}\crH^{k}_{\ov{p},c}(\rc X;R).$$
\item 
For any $k\leq \ov{p}(\tw)+1$, we have $\crH^k_{\ov{p},c}(\rc X;R)=0$.
\end{enumerate}
\end{proposition}

\begin{proof}
Recall $\rc X=\rc_{1}X=(X\times [0,1[)/(X\times\{0\})$.
From \propref{prop:homeo}, we may replace $\rc X$ by $\rc_{\infty} X=(X\times [0,\infty[)/(X\times\{0\})$  
for the determination of its blown-up cohomology with compact supports.
The pair $\cU=\{\rc_{1}X, X\times ]0,\infty[\}$ is
an open cover of  $\rc_{\infty} X$. 
The proof follows three steps.

\noindent
$\bullet$ \emph{Construction of an exact sequence.}
We consider the short exact sequence
\begin{equation}\label{equa:conecompact}
\xymatrix@1{
0\ar[r]&
\tN^{*,\cU}_{\ov{p},c}(\rc_{\infty} X)\ar[r]&
\tN^{*,\cU}_{\ov{p}}(\rc_{\infty} X)\ar[r]&
\frac{\tN^{*,\cU}_{\ov{p}}(\rc_{\infty} X)}{\tN^{*,\cU}_{\ov{p},c}(\rc_{\infty} X)}
\ar[r]&
0.
}
\end{equation}
For the study of the homology of the right-hand term, we introduce
$$\cG^*=
\left\{\omega\in \tN^*_{\ov{p}}(X\times ]0,\infty[)\mid
\exists a>0 \text{ such that  }
\omega_{\sigma}=0 \text{ if }\im\sigma\cap(X\times ]0,a[)=\emptyset\right\}.$$
The reduction $\tN^{*,\cU}_{\ov{p}}(\rc_{\infty} X)\to \tN^*_{\ov{p}}(X\times ]0,\infty[)$ induces a cochain map
$$\varphi\colon \frac{\tN^{*,\cU}_{\ov{p}}(\rc_{\infty} X)}{\tN^{*,\cU}_{\ov{p},c}(\rc_{\infty} X)}
\longrightarrow
\frac{\tN^*_{\ov{p}}(X\times ]0,\infty[)}{\cG^*}.
$$
First we show that $\varphi$ is an isomorphism. It is one-to-one because any cochain
$\omega\in \tN^{*,\cU}_{\ov{p}}(\rc_{\infty} X)$ such that $\varphi(\omega)=0$ 
owns as compact support the closed cone
$\tc_{a}X=X\times [0,a]/(X\times \{0\})$. 
For proving the surjectivity, we define an application $g\colon \rc_{\infty} X\to \{0,1\}$ 
by
$g([x,t])=0$ if $t\leq 1$ and $g([x,t])=1$ otherwise. 
We denote $\tilde{g}\in \tN^0_{\ov{0}}(\rc_{\infty} X)$
the 0-cochain associated to $g$ by \lemref{lem:0cochaine}.
Let
$\omega\in\tN^*_{\ov{p}}(X\times ]0,\infty[)$ be a cochain. 
For any regular simplex $\sigma\colon\Delta\to \rc_{\infty} X$, we set
\begin{equation}\label{equa:connectantcone}
\eta_{\sigma}=
\left\{
\begin{array}{cl}
0&
\text{if}\quad \im\sigma\subset\rc_{1}X,\\
\tilde{g}_{\sigma}\cup \omega_{\sigma}&
\text{if}\quad
\im\sigma\subset X\times ]0,\infty[.
\end{array}\right.
\end{equation}
If $\im\sigma\subset X\times ]0,1[$, we have $\tilde{g}_{\sigma}=0$
by construction of  $g$ and $\tilde{g}$.
Therefore $\eta\in \tN^{*,\cU}(\rc_{\infty} X)$ is well defined.
From $\tilde{g}\cup \omega\in \tN^*_{\ov{p}}(X\times ]0,\infty[)$, we deduce
 $\eta\in \tN^{*,\cU}_{\ov{p}}(\rc_{\infty} X)$. 
We are reduced to show $\omega-\varphi(\eta)\in\cG^*$.
For that, we choose $a=2$ and consider $\sigma\colon \Delta\to \rc_{\infty} X$ such that 
$\im\,\sigma\,\cap \,]0,2[=\emptyset$. By construction of $g$ and $\tilde{g}$,
we have $\tilde{g}_{\sigma}=1$, thus $\eta_{\sigma}=\tilde{g}_{\sigma}\cup \omega_{\sigma}=\omega_{\sigma}$.
The bijectivity of $\varphi$ is now established.

\medskip\noindent
$\bullet$ \emph{Proof of (a).} 
The homeomorphism $\R\cong ]0,\infty[$ and \lemref{lem:LRD} imply the acyclicity of $\cG^*$.
Thus, the right-hand term in the short exact sequence (\ref{equa:conecompact})
has for cohomology
$\crH^*_{\ov{p}}(X)$. 
From
\cite[Theorem A and Theorem B]{CST5},
we know that the complex $\tN^{*,\cU}_{\ov{p}}(\rc_{\infty} X)$ has for cohomology
$\crH^*_{\ov{p}}(\rc X)$ which has been determined in
\cite[Theorem E]{CST5}.
Thus, if $k\geq \ov{p}(\tw)+2$, the exact sequence associated to (\ref{equa:conecompact}) 
can be reduced to exact sequences of the form

$$0\to \crH^{k-1}_{\ov{p}}(X)
\xrightarrow{\delta_{1}}
\crH^{k}_{\ov{p},c}(\rc_{\infty} X)
\to
0,
$$
where $\delta_{1}$ is the connecting map determined by (\ref{equa:connectantcone}). 

\medskip\noindent
$\bullet$ \emph{Proof of (b).}
We first observe the commutativity of the diagram
$$\xymatrix{
\tN^*_{\ov{p}}(X\times ]0,\infty[)\ar[r]
&
\frac{\tN^*_{\ov{p}}(X\times ]0,\infty[)}{\cG^*}\\
\tN^{*,\cU}_{\ov{p}}(\rc_{\infty} X)\ar[r]\ar[u]
&
\frac{\tN^{*,\cU}_{\ov{p}}(\rc_{\infty} X)}{\tN^{*,\cU}_{\ov{p},c}(\rc_{\infty} X)}.
\ar[u]_-{\varphi}^-{\cong}
}$$
The top map is a quasi-isomorphism in all degrees, $\varphi$ is an isomorphism and the 
left-hand vertical map is a quasi-isomorphism if
$*\leq\ov{p}(\tw)$. Therefore, by using the determination
$\crH^{\ov{p}(\tw)+1}_{\ov{p}}(\rc_{\infty} X)=0$ (see \cite[Theorem E]{CST5}),
the map
$$\crH^{k}_{\ov{p}}(\rc_{\infty} X)\to 
H^k\left(\frac{\tN^{*,\cU}_{\ov{p}}(\rc_{\infty} X)}{\tN^{*,\cU}_{\ov{p},c}(\rc_{\infty} X)}\right)$$
is an isomorphism for any $k\leq \ov{p}(\tw)+1$.
The result follows.
\end{proof}

%%%%%%%%%%%%%%%%%%%%%%%%%%%%%%
\paragraph{\bf Cohomology with compact supports of the product with $\R$.}

\begin{proposition}\label{prop:xproduitRcompact}
Let $(X,\ov{p})$ be a locally compact and paracompact perverse space.
We denote also $\ov{p}$ the perversity induced on $X\times \R$.
Then, for any $k> 0$, there exists an isomorphism,
$$\crH^{k}_{\ov{p},c}(X;R)
\cong
\crH^{k+1}_{\ov{p},c}(X\times\R;R).
$$
\end{proposition}

\begin{proof}
With the notations of \lemref{lem:LRD}, we have a short exact sequence
with an acyclic middle term,
\begin{equation}\label{equa:LRK}
\xymatrix@1{
0\ar[r]&
\cL^*\cap\cR^*\cap\cK^*\ar[r]&
(\cL^*\cap \cK^*)\oplus (\cR^*\cap \cK^*)\ar[r]^-{\Phi}&
\cK^*\ar[r]&
0.
}
\end{equation}

Let $g\colon X\times \R\to\{0,1\}$ be the function defined by
$g(x,t)=0$ if $t\leq 1$ and $g(x,t)=1$ otherwise. Let $\tg$ be the associated cochain 
(see \lemref{lem:0cochaine}).
To any $\omega\in\cK^*$, we associate
$(\tg\cup \omega,(1-\tg)\cup\omega)\in (\cL^*\cap\cK^*)\oplus (\cR^*\cap\cK^*)$.
This  gives the surjectivity of $\Phi$ and 
determines the connecting map of the associated long exact sequence,
\begin{equation}\label{equa:connectant2}
[\omega]\mapsto \delta(\tg\cup\omega).
\end{equation}

\medskip\noindent
$\bullet$ \emph{The complex $\cK^*$ is quasi-isomorphic to $\tN^*_{\ov{p},c}(X)$.}
Denote $I_{0}\colon X\to X\times \R$ and $\pr\colon X\times \R\to X$
the maps defined by $I_{0}(x)=(x,0)$ and $\pr(x,t)=x$. 
They induce cochain maps,
$I_{0}^*\colon \tN^*_{\ov{p}}(X\times \R)\to \tN^*_{\ov{p}}(X)$
and
$\pr^*\colon \tN^*_{\ov{p}}(X)\to \tN^*_{\ov{p}}(X\times \R)$, (see \cite[Proposition 3.5]{CST5})
which verify $I_{0}^*(\cK^*)\subset \tN^*_{\ov{p},c}(X)$ and
$\pr^*(\tN^*_{\ov{p},c}(X))\subset \cK^*$. 
Thus we get
$$\xymatrix@1{
\cK^*\ar[r]^-{I_{0}^*}
&
\tN^*_{\ov{p},c}(X)\ar[r]^-{\pr^*}
&
\cK^*.
}$$
The map $I_{0}^*\circ \pr^*=(\pr\circ I_{0})^*$ is the identity map. 
The application
$\pr^*\circ I_{0}^*=(I_{0}\circ\pr)^*$ is homotopic to the identity map on
$\tN^*_{\ov{p}}(X\times \R)$,
see \cite[Theorem D]{CST5}.
From (\ref{equa:homotopiea}) applied to $\cK^*$, we deduce that 
$\pr^*\circ I_{0}^*$ is homotopic to the identity map on $\cK^*$.

\medskip\noindent
$\bullet$ \emph{The complex $\cK^*\cap \cL^*\cap \cR^*$ is 
quasi-isomorphic to  $\tN^*_{\ov{p},c}(X\times \R)$.}
We observe first that  $\tN^*_{\ov{p},c}(X\times \R)\subset \cK^*\cap \cL^*\cap \cR^*$
and denote by $\iota$ the corresponding canonical injection.
Let $\omega\in\cK^*\cap \cL^*\cap \cR^*$. There exist $a>0$ and a compact $K\subset X$ such that, 
if $\sigma\colon \Delta\to X\times \R$ satisfies \emph{one} 
of the following conditions, then we have $\omega_{\sigma}=0$,

$(i)$ $\im\sigma\subset X\times ]-\infty,a]$ or
$(ii)$ $\im\sigma\subset X\times [-a,\infty[$ or
$(iii)$ $(\im\sigma)\cap(K\times \R)=\emptyset$.

\smallskip
Choose an open subset $W$ of $X\times\R$ such that 
$K\times [-a,a]\subset W\subset \ov{W}$ and $\ov{W}$ compact. 
Set
$\cU_{\omega}=\left\{
X\times ]a,\infty[,\,
X\times ]-\infty,-a[,\,
(X\backslash K)\times \R
\right\}$.
From the properties of $\omega$, we deduce
$\omega\in\tN^{*,\cU_{\omega}}_{\ov{p},c}(X\times \R)$ with support $\ov{W}$.  
We have constructed a cochain map
$\psi$ which gives a commutative diagram with the quasi-isomorphism 
$\iota_{c}$ of \corref{cor:toutpetit},
$$\xymatrix{
\cK^*\cap \cL^*\cap \cR^*\ar[r]^-{\psi}
&
 \widetilde{\N}^*_{\ov{p},c}(X\times \R)\\
&
\tN^*_{\ov{p},c}(X\times \R).
\ar[u]_{\iota_{c}}^{\simeq}
\ar[ul]^-{\iota}
}$$
So, we get the injectivity of the homomorphism $\iota^*$ induced  by $\iota$. We establish now its surjectivity.
Let $\omega\in \cK^*\cap \cL^*\cap \cR^*$ of associated open cover $\cU_{\omega}$ and
such that $\delta\omega=0$.
Let $\sigma\colon \Delta\to X\times \R$ be a regular $\cU_{\omega}$-small simplex
such that $(\im\sigma)\cap \ov{W}=\emptyset$.
It follows: $\im\sigma\subset
(X\times ]a,\infty[)\cup (X\times ]-\infty,-a[)\cup ((X\backslash K)\times \R)$
and $\omega_{\sigma}=0$ by hypothesis on $\omega$.
Thus, from \propref{prop:Upetitscompact}, we get
$\rho_{\cU_{\omega}}(\omega)\in \tN^{*,\cU_{\omega}}_{\ov{p},c}(X\times \R)$
and from the proof of \propref{prop:Upetitscompact}, we deduce also
$(\varphi_{\cU_{\omega}}\circ\rho_{\cU_{\omega}})(\omega)\in
\tN^*_{\ov{p},c}(X\times\R)\subset 
\cK^*\cap \cL^*\cap \cR^*$
and
$
\omega-(\varphi_{\cU_{\omega}}\circ\rho_{\cU_{\omega}})(\omega)=\delta\Theta(\omega)$
with
$\Theta(\omega)\in \tN^*_{\ov{p},c}(X\times\R)\subset 
\cK^*\cap \cL^*\cap \cR^*
$. This proves the surjectivity of $\iota^*$.
\end{proof}

\begin{corollary}\label{cor:connectants}
Let $X$  be a compact filtered space.
The cone $\rc_{\infty} X$ is endowed with the conic filtration and with a perversity $\ov{p}$. 
We denote also $\ov{p}$ the induced perversity on $X\times ]0,\infty[$.
Then, for any $k\geq \ov{p}(\tw)+2$, the canonical injection
$I\colon X\times ]0,\infty[\hookrightarrow \rc_{\infty} X$ induces an isomorphism
$$\crH^k_{\ov{p},c}(X\times ]0,\infty[;R)
\xrightarrow{\cong}
\crH^k_{\ov{p},c}(\rc_{\infty} X;R)
.$$
\end{corollary}

\begin{proof}
In (\ref{equa:connectantcone}), 
to any $\omega\in\tN_{\ov{p}}(X\times ]0,\infty[)$, 
we associate a cochain
$\eta\in \tN^{*,\cU}_{\ov{p}}(\rc_{\infty} X)$.  
By precomposing with $\pr^*\colon \crH^*_{\ov{p}}(X)\to \crH^*_{\ov{p}}(X\times ]0,\infty[)$, 
we get the connecting map of (\ref{equa:conecompact}),
\begin{equation}\label{equa:connectant1}
\delta_{1}\colon \crH^*_{\ov{p}}(X)\to \crH^{*+1}_{\ov{p},c}(\rc_{\infty} X),
\end{equation}
defined by
$$ [\gamma]\mapsto\left\{\begin{array}{cl}
[\delta(\tilde{g}\cup \pr^*(\gamma))]& \text{if } \im\sigma\subset X\times ]0,\infty[,\\
0&\text{if } \im\sigma\subset  \rc_{1}X.
\end{array}\right.
$$
We observe that $\delta_{1}$ is an isomorphism if $*\geq \ov{p}(\tw)+1$.
In the proof of \propref{prop:xproduitRcompact}, with the same 0-cochain $\tilde{g}$,
we have specified in (\ref{equa:connectant2}) the connecting map of (\ref{equa:LRK}).
By precomposing with $\pr^*\colon \crH^*_{\ov{p}}(X)\to H^*(\cK^*)$, we get an isomorphism 
\begin{equation}\label{equa:connectant1b}
\delta_{2}\colon \crH^*_{\ov{p}}(X)\to H^{*+1}(\cL^*\cap \cR^*\cap \cK^*), [\gamma]\mapsto
[\delta(\tilde{g}\cup \pr^*(\gamma))].
\end{equation}
In the degrees of the statement, with $X$ compact, these two connecting maps are isomorphisms 
and give the following commutative diagram.
$$\xymatrix{
\crH^*_{\ov{p}}(X)\ar[r]^-{\delta_{1}}\ar[d]_{\delta_{2}}
&
\crH^{*+1}_{\ov{p},c}(\rc_{\infty} X)\ar[d]^{I^*}\\
H^{*+1}(\cL^*\cap \cR^*\cap \cK^*)
&
\crH^{*+1}_{\ov{p},c}(X\times ]0,\infty[).
\ar[l]_-{\iota}
}$$
Thus the homomorphism $I^*$ is an isomorphism.
\end{proof}

%%%%%%%%%%%%%%%%%
\subsection{Intersection cohomologies with compact supports}\label{sec:TWadroite}
In this section, we compare the blown-up intersection cohomology with an
intersection cohomology defined by the dual complex of intersection chains, in the case of compact supports. This second cohomology has already been introduced in \cite{zbMATH06243610} for the study of a duality in intersection homology via cap products.
We call it intersection cohomology with compact supports. 
We prove the existence of  an isomorphism between these two intersection cohomologies 
with compact supports
under an  hypothesis on the torsion, that we precise in the next definition.
We also mention  the existence of examples for which the two cohomologies differ.

\begin{definition}\label{def:localementlibre}
Let $R$ be a commutative ring and $\ov{p}$ a perversity on a CS set $X$.
The  CS set $X$ is \emph{locally $(\ov{p},R)$-torsion free}
if, for each singular stratum $S$ of link $L_{S}$, one has
$$\gT^{\ov{p}}_{\ov{q}(S)}(L_{S};R)=0,$$
where $\ov{q}(S)=\codim S-2-\ov{p}(S)$
and
$\gT^{\ov{p}}_{j}(L_{S};R)$ is the torsion 
$R$-submodule of $\gH^{\ov{p}}_{j}(L_{S};R)$.
\end{definition}

Note that the previous condition is always fulfilled if $R$ is a field. Also, in the case of a GM-perversity, 
that it is the torsion subgroup of $H^{\ov{p}}_{j}(L_{S};R)$ which is involved.
The existence of an isomorphism between the two cohomologies is based on the next result
 (cf. \cite{MR800845}, \cite{MR2210257}). 
 A proof of it can be found in \cite[Section 5.1]{FriedmanBook}.

\begin{proposition}\label{prop:superbredon}
Let $\cF_{X}$ be the category whose objects are homeomorphic 
in a filtered way to open subsets of a fixed CS set
$X$  and whose arrows are the inclusions and homeomorphisms respecting the filtration.
Let $\cAb_{*}$ be the category of graded abelian groups.
We consider two functors
$F^{*},\,G^{*}\colon \cF_{X}\to \cAb$
and a natural transformation
$\Phi\colon F^{*}\to G^{*}$,
such that the following properties are satisfied.
\begin{enumerate}[\rm (i)]
\item The functors $F^{*}$ and $G^{*}$ have  Mayer-Vietoris exact sequences 
and the natural transformation $\Phi$ induces a commutative diagram between these sequences.
\item If $\{U_{\alpha}\}$ is an increasing sequence of open subsets of $X$ 
and  $\Phi\colon F_{*}(U_{\alpha})\to G_{*}(U_{\alpha})$
is an isomorphism for each $\alpha$, then
$\Phi\colon F_{*}(\cup_{\alpha}U_{\alpha})\to G_{*}(\cup_{\alpha}U_{\alpha})$ is an isomorphism.
\item Let $L$ be a compact filtered space such that $\R^i\times \rc L$ is homeomorphic,
in a filtered way, to an open subset of $X$. If
$\Phi\colon F^{*}(\R^i\times (\rc L\backslash \{\tv\}))\to G^{*}(\R^i\times (\rc L\backslash \{\tv\}))$
is an isomorphism, then so is
$\Phi\colon F^{*}(\R^i\times \rc L)\to G^{*}(\R^i\times \rc L)$.
\item If $U$ is an open subset of $X$, included in only one stratum 
and homeomorphic to a Euclidean space, then  $\Phi\colon F^{*}(U)\to G^{*}(U)$
is an isomorphism.
\end{enumerate}
Then $\Phi\colon F^{*}(X)\to G^{*}(X)$ is an isomorphism.
\end{proposition}

If $(X,\ov{p})$ is a perverse space,
we set $\gC^*_{\ov{p}}(X;R)=\hom(\gC_*^{\ov{p}}(X;R),R)$
where $\gC_*^{\ov{p}}(X;R)$ is introduced in \defref{def:adroite}. 
The homology of $\gC^*_{\ov{p}}(X;R)$ is denoted $\gH^*_{\ov{p}}(X;R)$ (or $\gH^*_{\ov{p}}(X)$ 
if there is no ambiguity) and called 
\emph{$\ov{p}$-intersection cohomology.}
A cochain map 
$\chi\colon \tN^*_{\ov{p}}(X;R)\to \gC^*_{D\ov{p}}(X;R)$
can be defined as follows, see \cite[Proposition 13.4]{CST5}. 
If $\omega\in \tN^*(X;R)$ and if
 $\sigma\colon \Delta_{\sigma}=\Delta_{0}\ast\dots\ast\Delta_{n}\to X$
is a filtered simplex, we set:
$$\chi(\omega)(\sigma)=
\left\{\begin{array}{cl}
 \omega_{\sigma}(\tDelta_{\sigma})&
 \text{if } \sigma \text{ is regular,}\\
 0&
 \text{otherwise.}
 \end{array}\right.$$
For a field of coefficients and GM-perversities, we showed in \cite{CST1} that the map $\chi$ induces an isomorphism in homology. 
In \cite[Theorem F]{CST5}, this result is extended to the cases of  perversities 
defined at the level of each stratum,
 with coefficients in a Dedekind ring and 
 for any  paracompact, separable, locally $(D\ov{p},R)$-free CS set.
 More precisely, under the previous hypotheses, we prove 
\begin{equation}\label{equa:gmtw}
\crH^*_{\ov{p}}(X;R)\cong \gH^*_{D\ov{p}}(X;R).
\end{equation}
The next result is the adaptation of (\ref{equa:gmtw}) to cohomologies with compact supports,
with the definition (\cite{zbMATH06243610})
$\gH^*_{\ov{q},c}(X;R)=\lim_{K \text{\rm compact}}\gH^*_{\ov{q}}(X,X\backslash K;R)$.

\begin{proposition}\label{prop:TWadroite}
Let $(X,\ov{p})$  be a paracompact perverse CS set and $R$ a Dedekind ring. 
Denote $\ov{q}=D\ov{p}$. %
We suppose that one of the following hypotheses is satisfied.
\begin{enumerate}
\item The ring $R$ is a field.
\item The CS set $X$ is a  locally $({\ov{q}},R)$-torsion free pseudomanifold.
\end{enumerate}
Then there is an isomorphism
\begin{equation}\label{equa:GMdroiteTWc}
\crH^*_{\ov{p},c}(X;R)\cong \gH^*_{\ov{q},c}(X;R) 
\end{equation}
\end{proposition}

In the case of a GM-perversity $\ov{p}$, the conclusion of \propref{prop:TWadroite} can be stated as 
$$
\crH^*_{\ov{p},c}(X;R)\cong \lim_{K \text{compact}}H^*_{\ov{q}}(X,X\backslash K;R).
$$

\begin{proof}[Proof of \propref{prop:TWadroite}]
This proof is an adaptation of that of 
\cite[Theorem F]{CST5}.
Let $U$ be an open subset of $X$,
$\omega\in\tN^*_{\ov{p},c}(U)$ with compact support $K$ and
 $\sigma$  a regular filtered simplex.
From the construction of $\chi$, we  observe that
$\chi(\omega)({\sigma})\in\gC^*_{\ov{q}}(U,U\backslash K)$.
This gives a morphism
$\chi_U\colon \tN^*_{\ov{p},c}(U)\to  \lim_{K \text{compact}}\gC^*_{\ov{q}}(U,U\backslash K)$
which induces a natural transformation
$$\chi_U^*\colon \crH^*_{\ov{p},c}(U)\to  \lim_{K \text{compact}}\gC^*_{\ov{q}}(U,U\backslash K).
$$
For proving that $\chi^*=\chi_X^*$ is an isomorphism, we use \propref{prop:superbredon} whose  hypotheses are satisfied thanks to  \propref{prop:xproduitRcompact},  \corref{cor:connectants} and 
\cite[Chapter 7]{FriedmanBook}. 
\end{proof}

%%%%%%%%%%%%%%%
%%%%%%%%%%%%%%%%%%%%%%%%%%%%%
\section{Topological invariance. \thmref{thm:invariance}}\label{sec:invariance}

In this section, we prove the topological invariance of $\crH^*_{\ov{p},c}(-)$ in the case of 
GM-perversities and paracompact CS sets with no codimension one strata.
We first establish some additional properties of the blown-up cohomology with compact supports.
Later, for the proof of the topological invariance, we introduce a method developed by  King in \cite{MR800845}
and taken over with details and examples in \cite[Section 5.5]{FriedmanBook}.

\medskip
Let $U\subset X$ be an open subset of a perverse space $(X,\ov{p})$.
From \propref{prop:compactouvertdeX}, we deduce the existence of a short exact sequence 
given by the canonical surjective quotient map
$\pmb{R}_{U,X}$:
\begin{equation}\label{equa:relativecompact}
\xymatrix@1{
0\ar[r]&
\widetilde{\N}_{\ov{p},c}(U;R)\ar[rr]^{\pmb{I}_{U,X}}&&
\widetilde{\N}_{\ov{p},c}(X;R)\ar[rr]^-{\pmb{R}_{U,X}}&&
\widetilde{\N}_{\ov{p},c}(X,U;R)\ar[r]&
0.
}
\end{equation}
The homology of $\widetilde{\N}_{\ov{p},c}(X,U;R)$ is called
\emph{the relative blown-up cohomology with compact supports.}
The  long exact sequence associated to (\ref{equa:relativecompact}),
 \propref{prop:conecompact} and \propref{prop:xproduitRcompact} involve the next determination.

\begin{corollary}\label{prop:conerelative}
Let $X$ be an $n$-dimensional compact filtered space
and $\ov{p}$ be a GM-perversity.
The cone $\rc X$ is endowed with the induced filtration.
Then we have:
\begin{equation}\label{equa:conerelative}
\crH^j_{\ov{p},c}(\rc X,\rc X\backslash\{\tw\};R)=\left\{
\begin{array}{ccl}
0&\text{ if }&j\geq \ov{p}(n+1)+1,
\\
\crH^j_{\ov{p},c}(X;R)&\text{ if }&j< \ov{p}(n+1)+1.
\end{array}\right.
\end{equation}
\end{corollary}

The next result is an excision property.

\begin{corollary}\label{prop:Excisionhomologie}
Let $\ov{p}$ be a GM-perversity.
Let $X$ be a paracompact, locally compact filtered space, $F$ a closed subset of $X$  and $U$ an open subset of $X$ such that $F\subset U$.
Then the canonical inclusion $(X\backslash F,U\backslash F)\hookrightarrow (X,U)$ induces an isomorphism,
$$  \crH^{*}_{\ov{p},c}(X\backslash F,U\backslash F ;R) \cong \crH^{*}_{\ov{p},c}(X,U ;R).
$$
\end{corollary}

\begin{proof}
From the open covers $\{U,X\backslash F\}$ of $X$ and $\{U,U\backslash F\}$ of $U$, 
we obtain a commutative diagram between the associated Mayer-Vietoris exact sequences
(\propref{prop:MVcompact}).
$$
\xymatrix{
0\ar[r]&
\widetilde{\N}^*_{\ov{p},c}(U\cap (U\backslash F))
\ar[r]\ar[d]
&
\widetilde{\N}^*_{\ov{p},c}(U)\oplus \widetilde{\N}^*_{\ov{p},c}(U\backslash F)
\ar[r]\ar[d]
&
\widetilde{\N}^*_{\ov{p},c}(U)
\ar[r]\ar[d]
&
0\\
0\ar[r]&
\widetilde{\N}^*_{\ov{p},c}(U\cap (X\backslash F))
\ar[r]
&
\widetilde{\N}^*_{\ov{p},c}(U)\oplus \widetilde{\N}^*_{\ov{p},c}(X\backslash F)
\ar[r]
&
\widetilde{\N}^*_{\ov{p},c}(X)\ar[r]
&
0.
}$$
The Ker-Coker exact sequence gives an isomorphism
$\widetilde{\N}^*_{\ov{p},c}(X\backslash F,U\backslash F)\cong
\widetilde{\N}^*_{\ov{p},c}(X,U)$.
\end{proof}

We need also the 
blown-up cohomology with compact supports of a product with a sphere.

\begin{corollary}\label{cor:SfoisX}
Let $\ov{p}$ be a GM-perversity and $X$ a locally compact  and  paracompact filtered space. 
We denote
$S^{\ell}$ the sphere of $\R^{\ell+1}$ and endow the product ${S}^\ell\times X$
with the product filtration $({S}^\ell\times X)_i={S}^\ell\times X_i$.
Then, the projection $p_{X}\colon {S}^\ell \times X\to X$, $(z,x)\mapsto x$, 
induces isomorphisms 
$$\crH^{j}_{\ov{p},c}({S}^\ell\times X;R)\cong 
\crH^{j}_{\ov{p},c}(X;R)\oplus \crH^{j-\ell}_{\ov{p},c}(X;R).$$
\end{corollary}

\begin{proof}
Let $\{N,S\}$ be the two poles of the sphere $S^\ell$.
We do an induction in the Mayer-Vietoris exact sequence with
$U_1=X\times (S^\ell\backslash \{N\})$, $U_2=X\times (S^\ell\backslash \{S\})$
and
$U_1\cap U_2=X\times\R\times S^{\ell-1}$.
Propositions \ref{prop:MVcompact} and \ref{prop:xproduitRcompact}
conclude the proof.
\end{proof}

We briefly recall  King's construction. First, we say that two points $x_{0}$, $x_{1}$ of a topological space
are \emph{equivalent} if there exists a homeomorphism $h\colon (U_{0},x_{0}) \to  
(U_{1},x_{1})$
between two neighborhoods  of $x_{0}$ and $x_{1}$. We denote this relation by $\sim$.

Let $X$ be a CS set.  We observe that the equivalence classes of $\sim$ are union of strata.
We denote $X^*_{i}$ the union of the equivalence classes formed of strata of dimension 
less than or equal to $i$.
Let $X^*$ be the space $X$ endowed with this new filtration.
As $X^*$ is a CS set whose filtration does not depend on the initial filtration on $X$ (see \cite[Section 2.8]{FriedmanBook}), we have an intrinsic CS set associated to $X$.
The identity map as continuous application $\nu\colon X\to X^*$ is called \emph{intrinsic aggregation} of $X$.
In the next result we compare the blown-up intersection cohomology with compact supports of $X$ and $X^*$.

\begin{proposition}\label{prop:invariancestep}
Let $\ov{p}$ be a GM-perversity and $X$ a paracompact CS set 
with no codimension one strata.
We consider a stratum $S$ of $X$ and
a conic chart  $(U,\varphi)$  of $x\in S$.
If the intrinsic aggregation induces an isomorphism
 $$\nu_{*}\colon \crH_{\ov{p},c}^{*}(U\backslash S;R)\xrightarrow[]{\cong} 
 \crH^{*}_{\ov{p},c}((U\backslash S)^*;R),$$
then it induces also an isomorphism
$$\nu_{*}\colon \crH_{\ov{p},c}^{*}(U;R)\xrightarrow[]{\cong} \crH^{*}_{\ov{p},c}(U^*;R).$$
\end{proposition}

\begin{proof}
We may suppose $U = \R^k \times \rc W$, where $W$ is a compact filtered space and  $S \cap U= \R^k \times \{ \tw\}$.
From \cite[Lemma 2 and Proposition 1]{MR800845}, we deduce the existence of a homeomorphism of filtered spaces,
\begin{equation}\label{equa:homeo}
h\colon ( \R^k \times \rc W)^*\xrightarrow[]{\cong} \R^m\times \rc L,
\end{equation}
where $L$ is a (possibly empty) compact filtered space and $m\geq k$. 
Moreover $h$ satisfies,
\begin{equation}\label{equa:hetcone}
h(\R^k\times \{\tw\} ) \subset \R^m\times \{\tv\} \text{ and }
h^{-1}(\R^m\times \{\tv\})=\R^k\times \rc A,
\end{equation}
where $A$ is an $(m-k-1)$-sphere, $\tv$ and $\tw$ are the respective apexes of $\rc L$ and $\rc W$.
With these notations, the hypothesis and the conclusion of the statement become,
\begin{equation}\label{equa:hyp}
h \colon 
\crH^*_{\ov{p},c}(\R^k  \times \rc W \backslash (\R^k  \times \{\tw\})) 
\xrightarrow[]{\cong}
\crH^*_{\ov{p},c}(\R^m  \times \rc L \backslash h(\R^k  \times \{\tw\}))
\end{equation}
and
\begin{equation}\label{equa:conc}
h\colon 
\crH^*_{\ov{p},c}(\R^k  \times \rc W) 
\xrightarrow[]{\cong}
\crH^*_{\ov{p},c}(\R^m  \times \rc L).
\end{equation}
Set $s = \dim W$  and $t=\dim L$.  
The isomorphism $h$ of (\ref{equa:homeo}) implies $k+s=m+t$, and $s\geq t$ since $m\geq k$.

\medskip
$\bullet$ The result is direct if $s=-1$ and we may suppose $s\geq 0$ and $\R^k\times \{\tw\}$ a singular stratum. 
In fact, since $X$ has no strata of codimension~1, we have $s\geq 1$.

\medskip
$\bullet$ If $t=-1$, then $L=\emptyset$ and $\dim A=m-k-1=s$. We have a series of isomorphisms,
\begin{eqnarray}
\crH_{\ov{p},c}^j(\R^k\times \rc W)
&\cong&
\crH_{\ov{p},c}^j(\R^k\times \rc A)\cong_{(1)}
\crH_{\ov{p},c}^{j-k}(\rc A) 
\nonumber
\\ 
&\cong_{(2)}&
\left\{
\begin{array}{cl}
\crH^{j-k-1}_{\ov{p},c}(A)=H^{j-k-1}(A)
&\text{if } j-k-1\geq \ov{p}(s+1)+1,
\\
0&\text{if } j-k-1\leq \ov{p}(s+1),
\end{array}\right. \nonumber
\\
&\cong&
\left\{
\begin{array}{cl}
R
&\text{if } j=s+k+1,
\\
0&\text{otherwise.} 
\end{array}\right.
\label{equa:cas0}
\end{eqnarray}
The isomorphisms $\cong_{(1)}$ and $\cong_{(2)}$ arrive from 
\propref{prop:xproduitRcompact} and \propref{prop:conecompact} respectively. The last isomorphism is a consequence of
$$0<\ov{p}(s+1)+1\leq \ov{t}(s+1)+1 =s =\dim A.$$

\medskip
$\bullet$ We suppose now $t\geq 0$ and $s\geq 1$ and split the proof in two cases.\\
\emph{First case: suppose $j\leq \ov{p}(s+1)+1+k$.}
The same properties than above imply the two following series of isomorphisms.
\begin{equation}\label{equa:case1W}
\crH_{\ov{p},c}^j(\R^k\times \rc W)
\cong
\crH_{\ov{p},c}^{j-k}( \rc W)
\cong
\left\{
\begin{array}{cl}
\crH^{j-k-1}_{\ov{p},c}(W)
&\text{if } j-k-1\geq \ov{p}(s+1)+1,
\\
0&\text{if } j-k-1\leq \ov{p}(s+1),
\end{array}\right. 
\end{equation}
and 
\begin{equation}\label{equa:case1L}
\crH_{\ov{p},c}^j(\R^m\times \rc L)
\cong
\crH_{\ov{p},c}^{j-m}( \rc L)
\cong
\left\{
\begin{array}{cl}
\crH^{j-m-1}_{\ov{p},c}(L)
&\text{if } j-m-1\geq \ov{p}(t+1)+1,
\\
0&\text{if } j-m-1\leq \ov{p}(t+1).
\end{array}\right. 
\end{equation}
By using (\ref{equa:case1W}) and (\ref{equa:case1L}), the isomorphism
$\crH_{\ov{p},c}^j(\R^k\times \rc W)
\cong
\crH_{\ov{p},c}^j(\R^m\times \rc L)
\cong 0
$
is a consequence of the next inequalities where the first one results from \defref{def:perversite},
$$ \ov{p}(s+1)+k+1\leq \ov{p}(t+1)+s-t+k+1
\leq \ov{p}(t+1)+m+1.$$

\medskip\noindent
\emph{Second case: suppose $j> \ov{p}(s+1)+1+k$.}
We repeat the arguments used above in  two series of isomorphisms, 
together with additional properties
detailed below. First, we get
\begin{eqnarray}
\crH^j_{\ov{p},c}
(\R^m\times \rc L\backslash h(\R^k\times \{\tw\}))
&\cong&
\crH^j_{\ov{p},c}
(\R^k\times \rc W\backslash \R^k\times \{\tw\})
\cong
\crH^j_{\ov{p},c}
(\R^k\times (\rc W\backslash \{\tw\}))
\nonumber \\
&\cong&
\crH^{j-k-1}_{\ov{p},c}(W), \label{equa:leW}
\end{eqnarray}
where the first isomorphism is the hypothesis (\ref{equa:hyp}).
Denote $B\times\{\tv\}=h(\R^m\times \{\tw\})$.
We have also isomorphisms between the next relative cohomologies.
\begin{eqnarray}
\crH^j_{\ov{p},c}
(\R^m\times \rc L\backslash h(\R^k\times \{\tw\}),
\R^m\times \rc L\backslash (\R^m\times\{\tv\}))
&\cong_{(1)}& \nonumber\\
\crH^j_{\ov{p},c}
(\R^m\times \rc L\backslash (B\times \{\tv\}),
\R^m\times (\rc L\backslash \{\tv\}))
\cong
\crH^j_{\ov{p},c}
(\R^m\backslash B\times (\rc L,\rc L\backslash\{\tv\}))
&\cong& \nonumber\\
\crH^j_{\ov{p},c}
(\R^{k+1}\times A\times (\rc L,\rc L\backslash \{\tv\}))
\cong
\crH^{j-k-1}_{\ov{p},c}
(A\times (\rc L,\rc L\backslash \{\tv\}))
&\cong_{(2)}& \nonumber\\
\crH^{j-k-1}_{\ov{p},c}
(\rc L,\rc L\backslash \{\tv\})
\oplus
\crH^{j-k-1-\dim A}_{\ov{p},c}
(\rc L,\rc L\backslash \{\tv\}),\label{equa:case2relative}
\end{eqnarray}
where $\cong_{(1)}$ comes from the excision of $B\times (\rc L\backslash \{\tv\})$
(see \corref{prop:Excisionhomologie})
and $\cong_{(2)}$ from \corref{cor:SfoisX}.
In (\ref{equa:case2relative}), we observe from \corref{prop:conerelative} that the hypothesis
$j> \ov{p}(s+1)+1+k$ implies 
$\crH^{j-k-1}_{\ov{p},c}
(\rc L,\rc L\backslash \{\tv\})=0$.
Moreover, we have $j-k-1-\dim A=j-m$. Thus, 
with the restriction on $j$
imposed in this second case, the previous isomorphisms imply 
\begin{equation}\label{equa:case2relative2}
\crH^j_{\ov{p},c}
(\R^m\times \rc L\backslash h(\R^k\times \{\tw\}),
\R^m\times \rc L\backslash \R^m\times\{\tv\})
\cong 
\crH^{j-m}_{\ov{p},c}
(\rc L,\rc L\backslash \{\tv\}).
\end{equation}
In the next diagram,  left-hand arrows are a part of 
the long exact sequence of a pair and the horizontal isomorphisms
come successively from (\ref{equa:case2relative2}), (\ref{equa:leW}) and \propref{prop:xproduitRcompact}.
$$\xymatrix{
\crH^j_{\ov{p},c}
(\R^m\times \rc L\backslash h(\R^k\times \{\tw\},
\R^m\times \rc L\backslash \R^m\times\{\tv\})
\ar[d]\ar[r]^-{\cong}
&
\crH^{j-m}_{\ov{p},c}
(\rc L,\rc L\backslash \{\tv\})
\\
\crH^j_{\ov{p},c}
(\R^m\times \rc L\backslash h(\R^k\times \{\tw\})
\ar[d]\ar[r]^-{\cong}
&
\crH^{j-k-1}_{\ov{p},c}(W)
\\
\crH^j_{\ov{p},c}
(\R^m\times \rc L\backslash \R^m\times\{\tv\})
\ar[r]^-{\cong}
&
\crH^{j-m}_{\ov{p},c}
(\rc L\backslash \{\tv\})
\\
}$$
From this diagram and the long exact sequence associated to $(\rc L,\rc L\backslash \{\tv\})$, we deduce
\begin{equation}\label{equa:case2end}
\crH^{j-k-1}_{\ov{p},c}(W)
\cong
\crH^{j-m}_{\ov{p},c}
(\rc L).
\end{equation}
By using (\ref{equa:case2end}), the computation of the cohomology of a cone and 
the cohomology of a product with $\R$, we get
\begin{equation}\label{equa:case2endend}
\crH^j_{\ov{p},c}
(\R^m\times \rc L)
\cong
\crH^{j-m}_{\ov{p},c}
(\rc L)
\cong
\crH^{j-k-1}_{\ov{p},c}
(W)
\cong_{(1)}
\crH^{j-k}_{\ov{p},c}
(\rc W)
\cong
\crH^j_{\ov{p},c}
(\R^k\times \rc W),
\end{equation}
where $\cong_{(1)}$ is a consequence of the condition
$j-k\geq \ov{p}(s+1)+2$ imposed in this second case.
\end{proof}

Notice that the  hypothesis \emph{``with no codimension one strata"} is used in (\ref{equa:cas0}) where we assume that the sphere $A$ is of dimension $s>0$.

\medskip
The invariance property is deduced from \propref{prop:superbredon} 
applied to the natural transformation $\Phi_{U}\colon \crH^*_{\ov{p},c}(U)\to \crH^*_{\ov{p},c}(U^*)$.
All the ingredients being established, the proof goes as in \propref{prop:TWadroite} and 
we may leave it to the reader.

\begin{theorem}\label{thm:invariance}
Let $\ov{p}$ be a GM-perversity.
For any $n$-dimensional paracompact CS set $X$, 
with no codimension one strata, the intrinsic aggregation 
$\nu\colon X\mapsto X^*$
induces an isomorphism
$$\crH_{\ov{p},c}(X;R)\cong \crH_{\ov{p},c}(X^*;R).$$
\end{theorem}
%%%%%%%%%%%%%%%%%%%%%%%%%%%%%

%%%%%%%%%%%%%%%%%%

\section{Poincar\'e duality. \thmref{thm:dual}}\label{sec:dual}

\subsection{Intersection homology and Poincar\'e duality}\label{sec:GMcohomologie}

In this paragraph, $X$ is an oriented (\defref{def:orientation}) paracompact pseudomanifold and
$R$ is a commutative ring. We recall some known examples with the purpose 
of highlighting the
conditions of existence of a Poincar\'e duality in intersection homology.
First, Goresky and MacPherson display a bilinear form,
$H_{i}^{\ov{p}}(X;\Z)\times H_{n-i}^{\ov{t}-\ov{p}}(X;\Z)\to \Z$,
which becomes non degenerate after tensoring with the rationals,
cf. \cite{MR572580}.
By denoting $T^{\ov{p}}_{i}(-)$ the torsion subgroup of $H^{\ov{p}}_{i}(-)$, 
 M.~Goresky and P.~Siegel show in \cite[Theorem 4.4]{MR699009} that the previous bilinear form
 generates a non degenerate bilinear form,
$$T^{\ov{p}}_{i}(X)\times T^{\ov{t}-\ov{p}}_{n-i-1}(X)\to  \Q/\Z,$$
under the hypothesis of locally $(\ov{p},\Z)$-torsion free.
Without this additional hypothesis, the property disappears.
If we take as pseudomanifold $X$ the suspension of $\R P^3$ endowed with the perversity 
$\ov{p}$ taking the value  $1$ on the two apexes of the suspension,
we see that
$H_{2}^{\ov{p}}(X;\Z)=0$ and $H_{1}^{\ov{p}}(X;\Z)=\Z_{2}
$.

\medskip
We are interested now in the existence of a Poincar\'e duality
 given by a cap product
between intersection homology and cohomology groups.
We choose in this paragraph the intersection cohomology $H^*_{\ov{p}}(X;R)$ given by
$C^*_{\ov{p}}(X;R)=\hom(C_{*}^{\ov{p}}(X;R),R)$.
Even if we avoid the previous phenomenon of torsion by \emph{choosing a field $R$}, some restrictions
appear on the domain of values taken by the perversities.

\begin{example}\label{exam:suspensionprojectifcohomologie}
Consider a compact oriented manifold $M$ of dimension~$n-1$. 
We filter its suspension $X=\Sigma M$ by
$X_{0}=\{N,S\}= \dots =X_{n-1}\subset X_{n}=X$.
We choose a perversity
$\ov{p}$ such that $\ov{p}(\{N\})=\ov{p}(\{S\})=p$. 
The $\ov{p}$-intersection homology  of $X$ is determined for instance in \cite[Section 4.4]{FriedmanBook}
as
\begin{equation}\label{equa:homologiesuspension}
H_{i}^{\ov{p}}(X;R)=
\left\{
\begin{array}{ccl}
H_{i}(M;R)&\text{if}&i<n-p-1,\\
0&\text{if}&i= n-p-1 \text{ and } i\neq 0,\\
\tilde{H}_{i-1}(M;R)&\text{if}&i> n-p-1 \text{ and } i\neq 0,\\
R&\text{if}&0=i\geq n-p-1.
\end{array}\right.
\end{equation}
Even in the case of  field coefficients $R$, we observe the lack of duality if 
the perversity $\ov{p}$ does not lie between $\ov{0}$ and $\ov{t}$. For instance, 
with the space $X=\Sigma M$, we have:
\begin{itemize}
\item $H_0^{\ov{p}}(X)=R$ and $H^n_{D\ov{p}}(X)=0$ if $p>n-2=\ov{t}(n)$, 
\item $H_n^{\ov{p}}(X)=0$ and $H^0_{D\ov{p}}(X)=R$ if ${p}<0$.
\end{itemize}
In \thmref{thm:dual}, to overcome the restriction $\ov{p}\in [\ov{0},\ov{t}]$, 
we use the tame intersection homology recalled in \defref{def:adroite}.
\end{example}

% 

%%%%%%%%%%%%%%%%%%%%%
\subsection{Orientation of a pseudomanifold}

We recall the definition and properties of the orientation of pseudomanifolds, 
cf. \cite{MR572580} and  
\cite{zbMATH06243610}.

\begin{definition}\label{def:orientation}
An \emph{$R$-orientation} of a  pseudomanifold $X$ of dimension $n$
is an $R$-orientation of the manifold $X^n=X\backslash X_{n-1}$. 
For any  $x\in X^n$, we denote the associated local orientation class by
$\tto_{x}\in H_{n}(X^n,X^n\backslash\{x\};R)
=
\gH^{\ov{0}}_{n}(X,X\backslash\{x\};R)$

\end{definition}

\begin{theoremb}[{\cite{zbMATH06243610}}]\label{thm:fmclure}
Let $X$ be a pseudomanifold of dimension $n$, endowed with an $R$-orientation.
\begin{enumerate}[(1)]
\item If $X$ is normal,
the sheaf generated by
$U\to \gH^{\ov{0}}_{n}(X,X\backslash \ov{U};R)$
is constant and there exists a unique global section $s$ such that
$s(x)=\tto_{x}$ for any $x\in X^n$. Moreover for any $x\in X$,
$\gH_{i}^{\ov{0}}(X,X\backslash \{x\};R)=0$ if $i\neq n$
and
$\gH_{n}^{\ov{0}}(X,X\backslash \{x\};R)$
is the free $R$-module generated by $s(x)$. Henceforth we denote $\tto_{x}=s(x)$
for any $x\in X$.
\item If $X$ is not normal, we denote $\Pi\colon \hat{X}\to X$ the normalisation constructed by
G.~Pa\-dilla in \cite{MR2002448}
and we endow $\hat{X}$ with the $R$-orientation induced by the homeomorphism 
$\Pi\colon \hat{X}\backslash \hat{X}_{n-1}\cong X\backslash X_{n-1}$.
Then, we have $\gH_{i}^{\ov{0}}(X,X\backslash\{x\};R)=0$ if $i\neq n$ and
$\gH_{n}^{\ov{0}}(X,X\backslash \{x\};R)$
is the free $R$-module generated by $\{\Pi_{\ast}(\tto_{y})\mid y\in \Pi^{-1}(x)\}$. 
We denote
$\tto_{x}=\sum_{y\in\Pi^{-1}(x)}\Pi_{*}(\tto_{y})$.
\item For any compact $K\subset X$, there exists a unique element 
$\Gamma_{K}^X\in \gH_{n}^{\ov{0}}(X,X\backslash K;R)$
whose restriction equals $\tto_{x}$ for any $x\in K$.
The class $\Gamma^X_{K}$ is called \emph{the fundamental class of $X$ over $K$.}
If there is no ambiguity, we denote $\Gamma_{K}=\Gamma_{K}^X$.
\end{enumerate}
\end{theoremb}

\begin{remark}\label{rem:classefondamentaleetouvert}
Let $U\subset V\subset X$ be two open subsets of a  pseudomanifold $X$.
If  $K\subset U$ is a compact subset, the canonical inclusion $U\hookrightarrow V$
induces a homomorphism,
$I_{*}\colon \gH^{\ov{0}}_{n}(U,U\backslash K;R)\to
\gH^{\ov{0}}_{n}(V,V\backslash K;R)$. By construction, the fundamental class satisfies
\begin{equation}\label{equa:classeouvert}
I_{*}(\Gamma^U_{K})=\Gamma^V_{K}.
\end{equation}
\end{remark}

%%%%%%%%%%%%%%%%%%%
\subsection{The main theorem}\label{subsec:main} 
In this section, we prove that the cap product with the fundamental class of a pseudomanifold is the
isomorphism of Poincar\'e duality.

\begin{proposition}\label{prop:homocap}
Let $R$ be a commutative ring and $X$ an oriented  pseudomanifold of dimension~$n$,
endowed with a perversity $\ov{p}$. 
The  cap product with the fundamental class of $X$ defines a homomorphism,
\begin{equation}\label{equa:mapdedualite}
\cD\colon \crH_{\ov{p},c}^k(X;R)\to\gH_{n-k}^{\ov{p}}(X;R).
\end{equation}
\end{proposition}

\begin{proof} 
Let $\omega\in \tN^{k}_{\ov{p},c}(X)$ be a cocycle 
with compact support $K\subset X$. 
We choose a representing element $\gamma_{K}\in \gC^{\ov{0}}_{n}(X,X\backslash K)$ 
of the fundamental class
$\Gamma_{K}^X\in \gH_{n}^{\ov{0}}(X,X\backslash K)$. 
The differential of the chain $\omega\cap\gamma_{K}$ equals
$$\gd(\omega\cap\gamma_{K})=
(\delta\omega)\cap \gamma_{K}+(-1)^k \omega\cap (\gd \gamma_{K})=
(-1)^k \omega\cap (\gd \gamma_{K}).$$
The chain $\gamma_{K}$ being a relative  cycle, its differential satisfies 
$\gd \gamma_{K}\in \gC^{\ov{0}}_{n-1}(X\backslash K)$.
The subset $K$ being a support of $\omega$, we have 
$\omega\cap \gd\gamma_{K}=0$ and thus
$\omega\cap \gamma_{K}$ is a cycle in $\gC_{n-k}^{\ov{p}}(X)$.
Denote $[\omega\cap \gamma_{K}]\in \gH_{n-k}^{\ov{p}}(X)$ the 
associated tame intersection homology class. 
We have to prove  that this class does not depend on the choices made in its construction.
\begin{itemize} 
\item[$\bullet$]  \emph{The class $[\omega\cap\gamma_{K}]$ does not depend 
on the choice of the representing element $\gamma_{K}$ of $\Gamma_K$.}
 This is a consequence of the two following observations.
\begin{itemize}
\item If we replace $\gamma_{K}$ by $\gamma_{K}+\mu$ with 
$\mu \in \gC^{\ov{0}}_{n}(X\backslash  K)$, we have, 
by definition of the cap product,
$\omega\cap (\gamma_{K}+\mu)=
\omega \cap \gamma_{K}+\omega\cap \mu=\omega\cap\gamma_{K}$.
\item If we replace $\gamma_{K}$ by $\gamma_{K}+\gd\mu$ 
with $\mu\in \gC^{\ov{0}}_{n+1}(X,X\backslash K)$, 
the same argument implies
$[\omega\cap(\gamma_{K}+\gd \mu)]=[\omega\cap\gamma_{K}]$.
\end{itemize}
\item[$\bullet$]  \emph{The class $[\omega\cap\gamma_{K}]$ does not depend on the choice of the support
$K$ of $\omega$.}
 If $K$ and $L$ are two supports of $\omega$, we may suppose $L\subset K$.
 Therefore, the cycle $\gamma_{K}\in \gC_{n}^{\ov{0}}(X,X\backslash K)$
 is also a cycle in  $\gC_{n}^{\ov{0}}(X,X\backslash L)$.
 By uniqueness of the fundamental class over a compact,  the classes in 
 $\gH_{n}^{\ov{0}}(X,X\backslash L)$ associated to $\gamma_{L}$
 and $\gamma_{K}$ are equal.
Therefore, there exist
$\alpha\in\gC^{\ov{0}}_{n+1}(X,X\backslash L)$ and $\beta\in \gC^{\ov{0}}_{n}(X\backslash L)$
such that
 $\gamma_{K}-\gamma_{L}=\gd \alpha +\beta$.
 Then we may deduce,
 $$[\omega\cap\gamma_{K}]-[\omega\cap\gamma_{L}]=
 [\omega\cap\gd \alpha]+[\omega\cap\beta]=0.$$
 \item[$\bullet$] \emph{The class $[\omega\cap\gamma_{K}]$ does not depend on the choice of the cocycle
 $\omega$ in its associated cohomology class.} 
 This is a consequence of the Leibniz formula
$[(\omega+\delta\eta)\cap \gamma_{K}]=[\omega\cap\gamma_{K}\pm \gd(\eta\cap\gamma_{K})\pm
\eta\cap \gd\gamma_{K}]=[\omega\cap\gamma_{K}]$.
\end{itemize}
From \propref{prop:lecap}, we get the homomorphism $\cD$ of the statement. 
\end{proof}

\begin{theorem}\label{thm:dual}
Let $R$ be a commutative ring and $X$ an oriented  paracompact pseudomanifold 
of dimension~$n$, endowed with a perversity  $\ov{p}$. 
Then, the cap product with the fundamental class of $X$ induces an
isomorphism between the blown-up intersection cohomology with compact supports
and the tame intersection homology,
$$\cD\colon \crH^k_{\ov{p},c}(X;R)\xrightarrow{\cong}\gH_{n-k}^{\ov{p}}(X;R).$$
\end{theorem}

By using \propref{prop:TWadroite},  \thmref{thm:dual} gives also the duality theorem established by
Friedman and McClure in \cite{zbMATH06243610}, see also \cite{FriedmanBook}. 

\begin{corollary}\label{cor:dualGM3}
Let $R$ be a Dedekind ring and  $X$ an oriented paracompact pseudomanifold of  dimension~$n$, 
endowed with a perversity $\ov{p}$.
If $X$ is locally   $(D{\ov{p}},R)$-torsion free, the cap product with the fundamental class
induces an isomorphism,
$$\cD\colon \gH^k_{D\ov{p},c}(X;R)\xrightarrow{\cong} \gH_{n-k}^{\ov{p}}(X;R).$$
\end{corollary}

In the case of a compact pseudomanifold, we retrieve
the first result in this direction, proved by Goresky and MacPherson in \cite{MR572580}.

\begin{corollary}\label{cor:dualGM2}
Let  $X$ be an oriented compact pseudomanifold of dimension~$n$, with
no strata of codimension~1. 
For any GM-perversity $\ov{p}$, 
there exists an isomorphism,
$$\cD\colon H^k_{D\ov{p}}(X;\Q)\xrightarrow{\cong} H_{n-k}^{\ov{p}}(X;\Q).$$
\end{corollary}

\begin{proof}[Proof of \thmref{thm:dual}]
As any open subset $U\subset X$ is a pseudomanifold, we may consider the 
associated homomorphism defined in \propref{prop:homocap}, 
$\cD_{U}\colon \crH^k_{\ov{p},c}(U)\to \gH_{n-k}^{\ov{p}}(U)$.
If $U\subset V\subset X$ are two open subsets of $X$, the equality (\ref{equa:classeouvert}) 
gives the commutativity of the next diagram,
\begin{equation}\label{equa:dualnaturel}
\xymatrix{
\crH^k_{\ov{p},c}(V)
\ar[r]^-{\cD_{V}}
&
\gH^{\ov{p}}_{n-k}(V)\\
\crH^k_{\ov{p},c}(U)
\ar[r]^-{\cD_{U}}
\ar[u]^{I^*}
&
\gH^{\ov{p}}_{n-k}(U)
\ar[u]_{I_{*}}
}\end{equation}
where $I_{*}$ and $I^*$ are induced by the canonical inclusion $U\hookrightarrow V$
(see  \propref{prop:compactouvertdeX}).

The morphisms $\cD_{U}$ give a natural transformation between the functors 
$\crH_{\ov{p},c}^k(-)$
and
$\gH^{\ov{p}}_{n-k}(-)$
 and we  apply
 \propref{prop:superbredon} after having checked its hypotheses.

$\bullet$ Condition (ii) is direct and condition (iv) is the classical Poincar\'e 
duality theorem of manifolds. 

$\bullet$ By applying (\ref{equa:dualnaturel}) to 
$V=\R^i\times\rc L$ and $U=\R^i\times \rc L\backslash\{\tv\}\cong \R^i\times L\times ]0,\infty[$,
the condition (iii) comes from the properties of the blown-up cohomology
with compact supports established in
 \propref{prop:conecompact} and \corref{cor:connectants} 
 together with the properties of tame intersection homology recalled in the 
 Propositions   \ref{prop:RfoisXAdr} and \ref{prop:homologieconeAdr}.

$\bullet$ We consider now  condition (i).
The two theories, $\crH^*_{\ov{p},c}(-)$ and $\gH_{n-*}^{\ov{p}}(-)$,
have Mayer-Vietoris exact sequences, 
cf.  
\propref{prop:MVcompact}
and \thmref{thm:adroite}.
It is thus sufficient to prove that the map $\cD$ induces a commutative diagram between these two sequences.
This problem is reduced to two cases:
\begin{itemize}
\item a square with an open subset $U$ of a pseudomanifold $X$ and that is exactly the situation of
(\ref{equa:dualnaturel}),
\item a square containing the connecting maps of the two sequences and that we detail now. 
We consider the following diagram where $X=U_{1}\cup U_{2}$ and the maps $\delta_{c}$, $\delta_{h}$ are the connecting maps. 
\begin{equation}\label{equa:connectantsMV}
\xymatrix{
\crH^k_{\ov{p},c}(X)
\ar[rr]^-{\cD_{X}}
\ar[d]_{\delta_{c}}
&&
\gH^{\ov{p}}_{n-k}(X)
\ar[d]^{\delta_{h}}\\
\crH^{k+1}_{\ov{p},c}(U_{1}\cap U_{2})
\ar[rr]^-{\cD_{U_{1}\cap U_{2}}}
&&
\gH^{\ov{p}}_{n-k-1}(U_{1}\cap U_{2}).
}
\end{equation}
\end{itemize}
Let $\omega\in\tN^{k,\cU}_{\ov{p},c}(X)$ be a cocycle of compact support $K$.
For $i=1,\,2$, let
$g_{i}\colon X\to \{0,1\}$ be a partition of unity 
with $\Supp g_{i}\subset U_{i}$ and $\tg_{i}\in\tN^0_{\ov{0}}(X)$ the associated 0-cochain
defined in \lemref{lem:0cochaine}.
The connecting map
$\delta_{c}$ is constructed as follows in \propref{prop:MVcompact}:\\
\emph{we choose relatively compact open subsets, $W_{1}$, $W'_{1}$, $W_{2}$, $W'_{2}$ 
such that
$\Supp g_{i}\cap K\subset W'_{i}\subset \ov{W}'_{i}\subset W_{i}\subset \ov{W}_{i}\subset U_{i}$, 
and we define a cochain $\tg_{i}\cup \omega\in\tN^{k,\cU}_{\ov{p},c}(X)$ 
with compact support $\ov{W}'_{i}$, for $i=1,\,2$.
The open subset $W=W_{1}\cap W_{2}$ and the compact $F=\ov{W}_{1}\cap \ov{W}_{2}$ satisfy 
$$\Supp g_{1}\cap \Supp g_{2}\cap K\subset W
\subset F
\subset U_{1}\cap U_{2}.$$
We define also an open cover $\cW$ of $U_{1}\cap U_{2}$ and we set
$\delta_{c}([\omega])=[\delta\tg_{1}\cup \omega]$ where
$\delta\tg_{1}\cup \omega\in \tN^{*,\cW}_{\ov{p},c}(U_{1}\cap U_{2})$ 
is a cochain of compact support $F$.
By composing with the duality map, we get,}
\begin{equation}\label{equa:Ddel}
(\cD_{U_{1}\cap U_{2}}\circ \delta_{c})([\omega])=
[(\delta\tg_{1}\cup\omega)\cap \gamma_{F}],
\end{equation}
with $\gamma_{F}\in\gC_{n}^{\ov{0}}(X,X\backslash F)$ 
a representing element
of the fundamental class of $X$ over the compact $F$.
The compact $L=K\cup \ov{W}_{1}\cup \ov{W}_{2}$ is also a compact support of $\omega$.
From the open cover
$\{U_{1}\backslash \ov{W}_{1},U_{2}\backslash \ov{W}_{2}, U_{1}\cap U_{2}\}$
of $X$, 
by using properties of the subdivision process in intersection homology
(cf. \cite[Corollaire 7.13]{CST3}),
we decompose a representing element  $\gamma_{L}\in\gC_{n}^{\ov{0}}(X,X\backslash L)$ 
of the fundamental class of $X$ over the compact $L$ as
\begin{equation}\label{equa:decomposonsL} 
\gamma_{L}=\alpha_{1}+\alpha_{2}+\alpha_{12}+(\gd\gT^s+\gT^s\gd)(\gamma_{L}),
\end{equation}
where $s$ is an integer,  $\alpha_{i}\in \gC_{n}^{\ov{0}}(U_{i}\backslash \ov{W}_{i})$, $i=1,2$ and 
$\alpha_{12}\in \gC_{n}^{\ov{0}}(U_{1}\cap U_{2})$. 
The chain $\gamma_{L}$ is a relative cycle and by construction we have
$\gT^s\gd(\gamma_{L})\in \gC_{n}(X\backslash L)$.
The chains $\alpha_{1}$, $\alpha_{2}$ having a support in 
$X\backslash F$, we get 
$[\gamma_{L}]=[\alpha_{12}]\in 
\gH_{n}^{\ov{0}}(X,X\backslash F)$.
With \remref{rem:classefondamentaleetouvert} and $F\subset U_{1}\cap U_{2}$, we can choose
$[\alpha_{12}]\in \gH_{n}^{\ov{0}}(U_{1}\cap U_{2},U_{1}\cap U_{2}\backslash F)$
as fundamental class of  $U_{1}\cap U_{2}$ over $F$. 
Then, the equality (\ref{equa:Ddel}) becomes,
\begin{eqnarray}
\cD_{U_{1}\cap U_{2}}(\delta_{c}([\omega]))
&=&
[(\delta\tg_{1}\cup \omega)\cap \alpha_{12}]
=
 [(\delta(\tg_{1}\cup\omega))\cap  \alpha_{12}]
\nonumber \\
&=&
-(-1)^{|\omega|}[(\tg_{1}\cup \omega)\cap \gd\alpha_{12}]
\nonumber\\
&=_{(1)}&
-(-1)^{|\omega|}[(\tg_{1}\cup \omega)\cap \gd(\alpha_{1}+\alpha_{12})]
\nonumber\\
&=_{(2)}&
(-1)^{|\omega|} [(\tg_{1}\cup\omega)\cap \gd \alpha_{2}]\\
&=_{(3)}&
(-1)^{|\omega|} [\omega\cap \gd \alpha_{2}],
\label{equa:untour}
\end{eqnarray}
where 
\begin{itemize}
\item the equality (1) is a consequence of the fact that $\tg_{1}\cup\omega$ has for support $\ov{W}_{1}$
and $\alpha_{1}\in \gC_{n}^{\ov{0}}(U_{1}\backslash \ov{W}_{1})$,
\item the equality (2) comes from 
$\gd \alpha_{1}+\gd \alpha_{12}= \gd{\gamma_{L}}-\gd \alpha_{2}- \gd \gT^s\gd(\gamma_{L})$
and
$(\tg_{1}\cup\omega)\cap \gd \gamma_{L}=
(\tg_{1}\cup\omega)\cap \gd \gT^s\gd(\gamma_{L})
=0$,
because $\tg_{1}\cup \omega$ has for support $\ov{W}_{1}$ and 
$\gd\gamma_{L}\in \gC_{n-1}^{\ov{0}}(X\backslash L)$,
\item the equality (3) happens from $\tg_{1}+\tg_{2}=1$, from the fact that $\tg_{2}\cup\omega$ 
admits $\ov{W}_{2}$ as support and
from $\alpha_{2}\in \gC_{n}^{\ov{0}}(U_{2}\backslash \ov{W}_{2})$.
\end{itemize}

\medskip
We proceed now to the determination of $(\delta_{h}\circ \cD_{X})([\omega])$.
As the duality map does not depend on the choice of the support of $\omega$, 
we have, with the notations of (\ref{equa:decomposonsL}),
$$\cD_{X}([\omega])
=[\omega\cap \gamma_{L}]=
[\omega\cap\alpha_{2}+
\omega\cap (\alpha_{1}+\alpha_{12})],$$
with $\omega\cap\alpha_{2}\in \gC_{*}^{\ov{p}}(U_{2})$
and
$\omega\cap (\alpha_{1}+\alpha_{12})\in
\gC_{*}^{\ov{p}}(U_{1})$.
It follows
\begin{equation}\label{equa:unautretour}
\delta_{h}(\cD_{X}([\omega])
=
[\gd(\omega\cap\alpha_{2})]
=(-1)^{|\omega|} [\omega\cap \gd \alpha_{2}].
\end{equation}
The result is now a consequence of the equalities (\ref{equa:untour}) and (\ref{equa:unautretour}).
\end{proof}

%%%%%%%%%%%%%%%%%%%%
%%%%%%%%%%%%%%%%%%%%
\subsection{Poincar\'e duality and cup products} \label{subsec:dualcup}

In this paragraph, for a compact pseudomanifold, $X$, we present a nondegenerate pairing 
on the torsion free quotient of
$\crH^*_{\ov{\bullet}}(X;R)
$, defined  from the cup product. In particular, for Witt spaces, we get a nondegenerate pairing on the 
torsion free quotient of intersection cohomology
 for the middle perversity with coefficients in  $\Z$.
 
 \medskip
 For any $R$-module $A$, we denote by $\tors A$ the \emph{$R$-torsion submodule} of $A$
 and by $F\,A=A/\tors A$
 the \emph{$R$-torsion free quotient} of $A$. Recall also that a pairing $A\otimes B\to C$ is nondegenerate if the  
 adjunction maps, $A\to \Hom(B,C)$ and $B\to \Hom(A,C)$ are both injective.
 
 \begin{theorem}\label{thm:dualcup}
 Let $R$ be a Dedekind ring and
  $X$  a compact $n$-dimensional $R$-oriented pseudomanifold. For any perversity $\ov{p}$, there is a
 nondegenerate pairing arising from the cup product structure,
 $$\Phi\colon F\,\crH^k_{\ov{p}}(X;R)\otimes F\,\crH^{n-k}_{D\ov{p}}(X;R)\to R,$$
 defined by $\Phi([\omega]\otimes [\eta])=\varepsilon \cD([\omega]\cup [\eta])$,
 where $\cD(\omega\cup\eta)\in\gH_{0}^{\ov{t}}(X;R)$ is introduced in \propref{prop:homocap}
 and
 $\varepsilon\colon \gH_{0}^{\ov{t}}(X;R)\to R$ is the augmentation map.
 \end{theorem}

 \thmref{thm:dualcup} extends \cite[Proposition 8.4.7]{FriedmanBook} for $\partial X=\emptyset$, 
 in the sense that no hypothesis on 
 the torsion of the intersection homology of links  is required. 
 In particular, this result infers a nondegenerate pairing on $\Z$ for the middle perversity on Witt spaces; 
 we state it in \corref{cor:dualcup} after some recalls.

\begin{proof}
 The statement is equivalent to the fact that the cup product induces an injection
 \begin{equation}\label{equa:lephi}
 \Phi'\colon F\,\crH^k_{\ov{p},c}(X;R)\to \Hom (\crH^{n-k}_{D\ov{p}}(X;R),R),
 \end{equation}
 defined by $\Phi'([\omega])([\eta])=\Phi([\omega]\otimes [\eta])$.
 Let $Q(R)$ be the field of fractions of $R$.  
It suffices to prove the injectivity of  the morphism $\Phi'\otimes Q(R)$. 

From \cite[Theorem 6.3.25]{FriedmanBook},
we have $\gH_{*}^{\ov{p}}(X;Q(R))=\gH_{*}^{\ov{p}}(X;R)\otimes Q(R)$, which implies with
\thmref{thm:dual}, $\crH^{k}_{\ov{p}}(X;Q(R))\cong \crH^{k}_{\ov{p}}(X;R)\otimes Q(R)$. 
Now, classical arguments of homological algebra and \propref{prop:TWadroite} give the following isomorphisms
\begin{eqnarray*}
\Hom_{R}(\crH_{D\ov{p}}^{n-k}(X;R),R)\otimes Q(R)
&\cong&
\Hom_{R}(\crH_{D\ov{p}}^{n-k}(X;R),Q(R))\\
&\cong&
\Hom_{Q(R)}(\crH_{D\ov{p}}^{n-k}(X;R)\otimes Q(R),Q(R))\\
&\cong &
\Hom_{Q(R)}(\crH_{D\ov{p}}^{n-k}(X;Q(R)),Q(R))\\
&\cong &
\Hom_{Q(R)}(\gH_{\ov{p}}^{n-k}(X;Q(R)),Q(R))\\
&\cong &
\Hom_{Q(R)}(\Hom_{Q(R)}(\gH^{\ov{p}}_{n-k}(X;Q(R)),Q(R)),Q(R)).
\end{eqnarray*}
Denoting $\cB$ the bidual map, we establish now the commutativity up to sign of the following diagram,
$$\xymatrix{
\crH^k_{\ov{p},c}(X;Q(R))\ar[rr]^-{\Phi'\otimes Q(R)}\ar[drr]_-{\cD}&&
\Hom(\Hom(\gH_{n-k}^{\ov{p}}(X;Q(R)),Q(R)),Q(R))\\
&&
\gH_{n-k}^{\ov{p}}(X;Q(R)).
\ar[u]_{\cB}
}$$
Let $\omega\in \tN^{k}_{\ov{p}}(X;Q(R))$ be a cocycle and $\gamma\in \gC^{\ov{0}}_{n}(X;Q(R))$
a representing element  
of the fundamental class.
Let $\eta$ be a cocycle viewed as an element of $\Hom(\gH_{n-k}^{\ov{p}}(X;Q(R)),Q(R))$. By definition, we have
\begin{eqnarray*}
\cB(\cD([\omega]))([\eta])&=&
\pm [\eta](\cD([\omega])=\pm \eta(\omega\cap \gamma)
=
\pm \varepsilon (\eta\cap (\omega\cap \gamma))\\
&=&
\pm \varepsilon ((\eta\cup \omega)\cap \gamma)
= \pm \varepsilon \Phi'([\omega])([\eta]).
\end{eqnarray*}
As $\cB$ and $\cD$ are isomorphisms, $\Phi'\otimes Q(R)$ is one also. Thus the map 
$\Phi'\colon F\,\crH^k_{\ov{p}}(X;R)\to \Hom (\crH^{n-k}_{D\ov{p}}(X;R),R),
 $
  is injective.
  \end{proof}
  
Denote respectively by $\ov{m}$ and $D\ov{m}$
   the lower-middle perversity  and 
 the upper-middle perversity. They are defined on the singular strata by
 $$\ov{m}(S)=\left\lfloor \frac{(\codim S)-2}{2} \right\rfloor
 \text{ and }
D\ov{m}(S)=\left\lceil \frac{(\codim S)-2}{2} \right\rceil .
 $$
 
 \begin{definition}\label{def:witt}
 An \emph{$R$-Witt space} is a pseudomanifold $X$ such that, for any point $x$  in a stratum of 
 odd codimension, there is a link $L$ of $x$ such that
 $\gH^{\ov{m}}_{\frac{\dim L}{2}}(L;R)=0$.
 \end{definition}
 
 If $X$ is an  $R$-Witt space, then (see \cite[Proposition 9.1.8]{FriedmanBook} for instance), the inclusion of chain complexes induces an
 isomorphism
 \begin{equation}\label{equa:uplow}
 \xymatrix@1{
 \gH^{\ov{m}}_{*}(X;R)\ar[r]^-{\cong} &\gH^{D\ov{m}}_{*}(X;R).
 }
 \end{equation}
   
 \begin{corollary}\label{cor:dualcup}
 Let $X$ be a compact $\Z$-Witt space of dimension $n$. Then there are nondegenerate pairings
  \begin{eqnarray}
   F\,\crH_{\ov{m}}^k(X;\Z) \otimes F\,\crH_{\ov{m}}^{n-k}(X;\Z)
   &\to&
   \Z,\label{equa:pairing1}\\
    F\,\gH^{\ov{m}}_{k}(X;\Z) \otimes F\,\gH^{\ov{m}}_{n-k}(X;\Z)
    &\to&
    \Z.\label{equa:pairing2}
  \end{eqnarray}
 \end{corollary}
 
 \begin{proof}
 Formula (\ref{equa:pairing1}) is a consequence of \thmref{thm:dualcup} and (\ref{equa:uplow}). 
 The second formula, (\ref{equa:pairing2}),
 can be deduced from (\ref{equa:pairing1}) and \thmref{thm:dual}.
 \end{proof}
 
 If $X$ is a  locally $(\ov{m},R)$-torsion free Witt space, with $R$ a Dedekind ring, then 
 from the universal coefficients 
 formula for $\gH_{\ov{m}}^*(X;R)$ one  gets  nonsingular pairings 
 for the  $R$-torsion free quotient and the 
 $R$-torsion submodule of intersection  homology, see 
 \cite[Theorem 4.4]{MR699009} or \cite[Proposition 9.2.3]{FriedmanBook}.

\begin{example}
We detail now an example of a compact pseudomanifold for which the map $\Phi'$ of  (\ref{equa:lephi}) is not surjective.

Let $E$ be the tangent space of the sphere $S^2$. We denote by $D_{E}\to S^2$ and
$S_{E}\to S^2$ the disk and sphere bundles associated to $E$, respectively.
The Thom space $X$ associated to $E$ is obtained by gluing a cone on $S_{E}$ to $D_{E}$, i.e.,
$X=D_{E}\cup_{S_{E}}\tc S_{E}$. Let $\omega_{2}\in H^2(S^2;\Z)$ be the orientation class, related to
the Euler class $e$ of $E$ by
$e=\chi(S^2)\omega_{2}=2\omega_{2}$. We denote by $\theta\in H^2(X;\Z)$ the Thom class of $E$,
recalling that the cup product with $\theta$ gives the Thom isomorphism
$\phi\colon H^k(S^2;\Z)\cong H^{k+2}(X;\Z)$. 

We stratify $X$
with the unique stratum $\{\tv\}$ formed of the apex of the cone $\tc S_{E}$ and get a structure
of oriented pseudomanifold on $X$.
A perversity on $X$ is entirely determined by its value $k$ on the stratum $\{\tv\}$ and denoted by $\ov{k}$.
The top perversity is $\ov{2}$ and the middle perversity $\ov{1}$.
The stratum $\{\tv\}$ is of codimension 4 and its link is $S_{E}$;  thus
the Thom space is a $\Z$-Witt space.

We determine the blown-up cohomology from the Mayer-Vietoris sequence (\cite[Theorem~C]{CST5}) 
associated to an open
cover $\{U,V\}$ of $X$ with $U$ of the homotopy type of $S^2$, $V=\rc S_{E}$ and $U\cap V$
of the homotopy type of $S_{E}=\R P^3$.
Straightforward but easy computations give
$\crH^j_{\ov{0}}(X;\Z)\cong \crH^j_{\ov{1}}(X;\Z)$, 
\begin{equation}\label{equa:thomS21}
\crH^2_{\ov{1}}(X;\Z)= \Z\cong\Ker(H^2(S^2;\Z)\to H^2(S_{E};\Z)),
\end{equation}
and
\begin{equation}\label{equa:thomS22}
\crH^4_{\ov{2}}(X;\Z)
=
\Z\cong H^3(S_{E};\Z)\cong H^2(S^2;\Z).
\end{equation}
The map $H^2(S^2;\Z)=\Z\to H^2(S_{E};\Z)=\Z_{2}$ 
appearing in (\ref{equa:thomS21}) 
is induced by the fibration
$S_{E}\to S^2$. From the Gysin sequence, we know that is the canonical surjection for the quotient of $\Z$ by $2\Z$.
Thus $\crH^2_{\ov{1}}(X;\Z)$ injects in $H^2(S^2;\Z)=\Z$ as $2\Z\hookrightarrow \Z$ and we want to study the pairing
\begin{equation}\label{equa:pairingS2}
\Phi\colon \crH^2_{\ov{1}}(X;\Z)\otimes \crH^2_{\ov{1}}(X;\Z)
\xrightarrow{-\cup-}
 \crH^4_{\ov{2}}(X;\Z)\cong H^2(S^2;\Z)
\xrightarrow{-\cap \omega_{2}} \Z.
\end{equation}
Recall from \cite[Proposition 13.5]{CST5} the existence of an isomorphism of algebras,
$\crH_{\ov{0}}^*(X;\Z)\cong H^*(X;\Z)$ which, together with the isomorphism
$\crH_{\ov{0}}^*(X;\Z)\cong \crH_{\ov{1}}^*(X;\Z)$, reduces the pairing (\ref{equa:pairingS2}) 
to the determination of $\theta^2
\in H^4(X;\Z)$.
\cite[Proposition 17.7.6]{MR1249482} specifies this square through the Thom isomorphism $\phi$ as
$\phi^{-1}(\theta^2)=e=2\omega_{2}$. This implies $\Phi(\theta\otimes\theta)=2$ and the nonsurjectivity of
the adjoint map
$ \Phi'\colon \crH^2_{\ov{1}}(X;\Z)\to \Hom (\crH^{2}_{\ov{1}}(X;\Z),\Z)
$.
\end{example}

%%%%%%%%%%%%%%%%%%%
%%%%%%%%%%%%%%%%%%%
\providecommand{\bysame}{\leavevmode\hbox to3em{\hrulefill}\thinspace}
\providecommand{\MR}{\relax\ifhmode\unskip\space\fi MR }
% \MRhref is called by the amsart/book/proc definition of \MR.
\providecommand{\MRhref}[2]{%
  \href{http://www.ams.org/mathscinet-getitem?mr=#1}{#2}
}
\providecommand{\href}[2]{#2}

%%
%%%

\begin{thebibliography}{10}

\bibitem{MR2662593}
Markus Banagl, \emph{Intersection spaces, spatial homology truncation, and
  string theory}, Lecture Notes in Mathematics, vol. 1997, Springer-Verlag,
  Berlin, 2010. \MR{2662593}

\bibitem{MR1171153}
Jean-Paul Brasselet, Gilbert Hector, and Martin Saralegi, \emph{{${\mathcal
  L}^2$}-cohomologie des espaces stratifi\'es}, Manuscripta Math. \textbf{76}
  (1992), no.~1, 21--32. \MR{1171153 (93i:58009)}

\bibitem{CST1}
David {Chataur}, Martintxo {Saralegi-Aranguren}, and Daniel {Tanr{\'e}},
  \emph{{Intersection Cohomology. Simplicial blow-up and rational homotopy.}},
  ArXiv Mathematics e-prints 1205.7057 (2012), To appear in Mem. Amer. Math.
  Soc.

\bibitem{CST3}
\bysame, \emph{{Homologie d'intersection. Perversit\'es g\'en\'erales et
  invariance topologique}}, ArXiv Mathematics e-prints 1602.03009 (2016).

\bibitem{CST2}
\bysame, \emph{Steenrod squares on intersection cohomology and a conjecture of
  {M} {G}oresky and {W} {P}ardon}, Algebr. Geom. Topol. \textbf{16} (2016),
  no.~4, 1851--1904. \MR{3546453}

\bibitem{CST5}
\bysame, \emph{{Blown-up intersection cohomology}}, ArXiv Mathematics e-prints
  1701.00684 (2017), To appear in Contemporary Math.

\bibitem{CST7}
\bysame, \emph{Singular decompositions of a cap product}, Proc. Amer. Math.
  Soc. \textbf{145} (2017), no.~8, 3645--3656. \MR{3652815}

\bibitem{FriedmanBook}
Greg Friedman, \emph{Singular intersection homology}, Available at\\
  http://faculty.tcu.edu/gfriedman/index.html.

\bibitem{MR2276609}
\bysame, \emph{Singular chain intersection homology for traditional and
  super-perversities}, Trans. Amer. Math. Soc. \textbf{359} (2007), no.~5,
  1977--2019 (electronic). \MR{2276609 (2007k:55008)}

\bibitem{MR2507117}
\bysame, \emph{Intersection homology and {P}oincar\'e duality on homotopically
  stratified spaces}, Geom. Topol. \textbf{13} (2009), no.~4, 2163--2204.
  \MR{2507117 (2010d:55010)}

\bibitem{MR2529162}
\bysame, \emph{On the chain-level intersection pairing for {PL}
  pseudomanifolds}, Homology, Homotopy Appl. \textbf{11} (2009), no.~1,
  261--314. \MR{2529162 (2010m:55004)}

\bibitem{MR2721621}
\bysame, \emph{Intersection homology with general perversities}, Geom. Dedicata
  \textbf{148} (2010), 103--135. \MR{2721621}

\bibitem{MR2796412}
\bysame, \emph{An introduction to intersection homology with general perversity
  functions}, Topology of stratified spaces, Math. Sci. Res. Inst. Publ.,
  vol.~58, Cambridge Univ. Press, Cambridge, 2011, pp.~177--222. \MR{2796412
  (2012h:55007)}

\bibitem{2016arXiv160905975F}
\bysame, \emph{{The chain-level intersection pairing for PL pseudomanifolds
  revisited}}, ArXiv e-prints (2016).

\bibitem{zbMATH06243610}
Greg {Friedman} and James~E. {McClure}, \emph{{Cup and cap products in
  intersection (co)homology.}}, {Adv. Math.} \textbf{240} (2013), 383--426.

\bibitem{MR572580}
Mark Goresky and Robert MacPherson, \emph{Intersection homology theory},
  Topology \textbf{19} (1980), no.~2, 135--162. \MR{572580 (82b:57010)}

\bibitem{MR696691}
\bysame, \emph{Intersection homology. {II}}, Invent. Math. \textbf{72} (1983),
  no.~1, 77--129. \MR{696691 (84i:57012)}

\bibitem{MR1014465}
Mark Goresky and William Pardon, \emph{Wu numbers of singular spaces}, Topology
  \textbf{28} (1989), no.~3, 325--367. \MR{1014465}

\bibitem{MR699009}
Mark Goresky and Paul Siegel, \emph{Linking pairings on singular spaces},
  Comment. Math. Helv. \textbf{58} (1983), no.~1, 96--110. \MR{699009
  (84h:55004)}

\bibitem{MR1249482}
Dale Husemoller, \emph{Fibre bundles}, third ed., Graduate Texts in
  Mathematics, vol.~20, Springer-Verlag, New York, 1994. \MR{1249482}

\bibitem{MR800845}
Henry~C. King, \emph{Topological invariance of intersection homology without
  sheaves}, Topology Appl. \textbf{20} (1985), no.~2, 149--160. \MR{800845
  (86m:55010)}

\bibitem{MacPherson90}
R.~MacPherson, \emph{Intersection homology and perverse sheaves}, Unpublished
  AMS Colloquium Lectures, San Francisco, 1991.

\bibitem{MR0440554}
John~W. Milnor and James~D. Stasheff, \emph{Characteristic classes}, Princeton
  University Press, Princeton, N. J.; University of Tokyo Press, Tokyo, 1974,
  Annals of Mathematics Studies, No. 76. \MR{0440554}

\bibitem{MR2002448}
G.~Padilla, \emph{On normal stratified pseudomanifolds}, Extracta Math.
  \textbf{18} (2003), no.~2, 223--234. \MR{2002448 (2004h:57031)}

\bibitem{MR1245833}
Martin Saralegi, \emph{Homological properties of stratified spaces}, Illinois
  J. Math. \textbf{38} (1994), no.~1, 47--70. \MR{1245833 (95a:55011)}

\bibitem{MR2210257}
Martintxo Saralegi-Aranguren, \emph{de {R}ham intersection cohomology for
  general perversities}, Illinois J. Math. \textbf{49} (2005), no.~3, 737--758
  (electronic). \MR{2210257 (2006k:55013)}

\bibitem{MR0319207}
L.~C. Siebenmann, \emph{Deformation of homeomorphisms on stratified sets. {I},
  {II}}, Comment. Math. Helv. \textbf{47} (1972), 123--136; ibid. 47 (1972),
  137--163. \MR{0319207 (47 \#7752)}

\end{thebibliography}
\end{document}